\documentclass[a4paper, 11pt, reqno]{amsart}
\usepackage{amsmath, amsthm, amscd, amssymb, amsfonts, amsxtra, amssymb, latexsym, bm}
\usepackage{enumerate}
\usepackage{verbatim}
\usepackage{graphicx,epsfig,tikz}
\usepackage{float}
\usepackage{hyperref}
\hypersetup{colorlinks = true,	allcolors  = blue}

\newcommand{\Z}{\mathbb{Z}}
\newcommand{\N}{\mathbb{N}}
\newcommand{\ff}{\mathbb{F}}
\newcommand{\Tr}{\operatorname{Tr}}

\newcommand{\CC}{\mathcal{C}}

\newcommand{\G}{\Gamma}

\newcommand{\sk}{\smallskip}
\newcommand{\msk}{\medskip}

\newtheorem{thm}{Theorem}[section]
\newtheorem{prop}[thm]{Proposition}
\newtheorem{lem}[thm]{Lemma}
\newtheorem{coro}[thm]{Corollary}

\theoremstyle{definition}
\newtheorem{rem}[thm]{Remark}
\newtheorem{exam}[thm]{Example}
\newtheorem{defi}[thm]{Definition}

\theoremstyle{remark}

\usepackage{color}

\begin{document} \sloppy
\numberwithin{equation}{section}
\title[]{Spectral properties of generalized Paley graphs}
\author{Ricardo A.\@ Podest\'a, Denis E.\@ Videla}
\dedicatory{\today}
\thanks{Partially supported by CONICET, FonCyT and SECyT-UNC}
\address{Ricardo A.\@ Podest\'a. FaMAF -- CIEM (CONICET), Universidad Nacional de C\'ordoba. 
	\newline Av.\@ Medina Allende 2144, Ciudad Universitaria, (5000), C\'ordoba, Argentina. \newline
{\it E-mail: podesta@famaf.unc.edu.ar}}
\address{Denis E.\@ Videla. FaMAF -- CIEM (CONICET), Universidad Nacional de C\'ordoba. 
\newline	Av.\@ Medina Allende 2144, Ciudad Universitaria, (5000), C\'ordoba, Argentina.
	\newline {\it E-mail: dvidela@famaf.unc.edu.ar}}

\begin{abstract}
We study the spectrum of generalized Paley graphs $\G(k,q)=Cay(\ff_q,R_k)$, undirected or not, with $R_k=\{x^k:x\in \ff_q^*\}$ where $q=p^m$ with $p$ prime and $k\mid q-1$. 
We first show that the eigenvalues of $\G(k,q)$ are given by the Gaussian periods $\eta_{i}^{(k,q)}$ with $0\le i\le k-1$. 
Then, we explicitly compute the spectrum of $\G(k,q)$ with $1\le k \le 4$ and of $\G(5,q)$ for $p\equiv 1\pmod 5$ and $5\mid m$. 
Also, we characterize those GP-graphs having integral spectrum, showing that $\G(k,q)$ is integral if and only if $p$ divides $(q-1)/(p-1)$. 
Next, we focus on the family of semiprimitive GP-graphs. We show that they are integral strongly regular graphs  (of pseudo-Latin square type).
Finally, we characterize all integral Ramanujan graphs $\G(k,q)$ with $1\le k \le 4$ or where $(k,q)$ is a semiprimitive pair.
\end{abstract}

\maketitle

\section{Introduction}
In this paper we study the spectrum of generalized Paley graphs (GP-graphs for short), and some properties that can be deduced from the spectrum. 
The work has three parts. We first study the spectrum of GP-graphs $\G(k,q)$ and put the spectrum in terms of cyclotomic Gaussian periods. This allows us to give $\mathrm{Spec}(\G(k,q))$ explicitly for $1\le k\le 4$ and to characterize those GP-graphs having integral spectrum (first main result).
In the second part, we focus on the family of semiprimitive GP-graphs. These graphs are, in particular, strongly regular graphs with integral spectrum. We study the spectrum, parameters and invariants of these graphs as strongly regular graphs. Finally, in the third part, we study GP-graphs which are Ramanujan. We classify all integral Ramanujan graphs of the form $\G(k,q)$ with $1\le k\le 4$ and all semiprimitive GP-graphs which are Ramanujan (second main result).

Some results on the spectrum of arbitrary GP-graphs and on the structure of semiprimitive GP-graphs are known, and can be found scattered in the literature. For completeness, we have decided to include our proofs (with some additional references) to give a unified treatment and notations to these topics. However, the explicit computation of the spectra for $\G(3,q)$ and $\G(4,q)$, the characterization of GP-graphs with integral spectrum and the classification of Ramanujan semiprimitive GP-graphs 
are completely new.

\paragraph{\textit{Generalized Paley graphs}}
If $G$ is an abelian group and $S$ is a subset of $G$ not containing $0$, the associated Cayley graph $\Gamma = X(G,S)$ is the directed graph (digraph) with vertex set $G$ and where two vertices $u,v$ form a directed edge from $u$ to $v$ in $\Gamma$ if and only if $v-u \in S$. Since $0\notin S$ then $\Gamma$ has no loops. 
Analogously, the Cayley sum graph $X^+(G,S)$ has the same vertex set $G$ but now $v,w\in G$ are connected in $\Gamma$ 
by an arrow from $v$ to $w$ if and only if $v+w \in S$. 
We will use the notation $X^*(G,S)$ when we want to consider both $X(G,S)$ and $X^+(G,S)$ indistinctly.
Notice that if $S$ is symmetric, that is $-S=S$, then 
$X^*(G,S)$ is an $|S|$-regular simple (undirected without multiple edges) graph. 
Actually, given any two vertices $u,v$ there are two directed edges, $\vec{uv}$ and $\vec{vu}$. As usual, we consider these two directed edges as a non-directed single one denoted $uv$. 
However, the graph $X^+(G,S)$ may contain loops. In this case, there is a loop on vertex $x$ provided that  $x+x \in S$.

The \textit{generalized Paley graph} and  \textit{generalized Paley sum graph} are the Cayley graphs respectively given by
\begin{equation*} \label{Gammas}
	\G(k,q) = X(\ff_{q},R_{k}) \quad \text{and} \quad \Gamma^+(k,q) = X^+(\ff_q,R_k) 
\end{equation*} 
with connection set 
	$$R_{k} = \{ x^{k} : x \in \ff_{q}^*\}.$$ 
That is, $\G(k,q)$ is the graph with vertex set $\ff_{q}$ and two vertices $u,v \in \ff_{q}$ are neighbours (directed edge) 
if and only if $v-u=x^k$ for some $x\in \ff_q^*$. We will refer to them simply as \textit{GP-graphs} and \textit{GP$^+$-graphs} respectively (or \textit{GP$^*$-graphs} for both indistinctly).

Notice that if $\omega$ is a primitive element of $\ff_{q}$, then 
	$R_{k} = \langle \omega^{k} \rangle = \langle \omega^{(k,q-1)} \rangle$. 
This implies that $\G(k,q) = \G(k',q)$, where 
	$$k'=\gcd(k,q-1),$$
and that $\G(k,q)$ is a $\frac{q-1}{k'}$-regular graph. 
Thus, one usually assumes that 
	$$k \mid q-1$$ 
(hence $k'=k$), for if not we have that $\G(k,q)=\G(1,q)=K_q$.
Summing up, we have 
	$$\G(k,q) = \begin{cases}
		    \hfil K_q 		& \qquad \text{if $k'=1$}, \\[1mm]
		    \G(k',q)		& \qquad \text{if $k'>1$}.
	\end{cases}$$
Notice that for $q$ even, we have that $\Gamma^+(k,q)=\Gamma(k,q)$. 
On the other hand, when $q$ is odd, one can show that $\Gamma^+(k,q)$ has loops, since in this case there are exactly $|R_k|$ elements $x\in \ff_{q}$ such that $x+x=2x \in R_{k}$ (multiplication by $2$ is a bijection in $\ff_{q}$ for $q$ odd). 

The graph $\G(k,q)$ is undirected if and only if $q$ is even or else $q$ is odd and $k \mid \tfrac{q-1}2$.
The graph is connected  if and only if $\tfrac{q-1}{k}$ is a primitive divisor of $q-1$ (i.e.\@ $\frac 1k(p^m-1)$ does not divide $p^{a}-1$ for any $a<m$, where $q=p^m$).
This was proved in \cite{LP} for the undirected case, 
but it also holds in the directed case since $\G(k,q)$ is strongly connected if and only if the Waring number $g(k,q)$ exists, and this happens if and only if $\frac{q-1}{k}$ is a primitive divisor of $q-1$. We recall that a strongly connected digraph is a directed graph in which there is a directed path in each direction between any pair of vertices of the graph. 

For some values of $k$ and $q$, the GP-graphs $\G(k,q)$ are known graphs. For instance, for $k=1, 2$ we get the complete graph $\G(1,q)=K_q$, the classic (undirected) Paley graph $\Gamma(2,q) = P(q)$ for $q\equiv 1 \pmod 4$, and the directed Paley graph $\vec P(q)$ for $q\equiv 3 \pmod 4$. The graphs $\G(3,q)$ and $\G(4,q)$ are of interest too (see \cite{PV3}, where infinite pairs of equienergetic non-isospectral regular graphs ${\Gamma(k,q), \bar \Gamma(k,q)}$, $k=3,4$, are obtained).
One can see that for $p$ prime, we have that $\G(\frac{p-1}2,p)=C_p$ and $\G(p-1,p) = \vec{C_p}$, where $C_p$ and $\vec{C_p}$ are the undirected and directed $p$-cycles, respectively. 
Generalized Paley graphs with $k=q^\ell+1$ are studied in \cite{PV1}.
The connected GP-graphs of the form $\G(\tfrac{p^{bm}-1}{b(p^m-1)}, p^{bm})$ are the Hamming graphs $H(b,p^{m})$ (see \cite{LP}).

\paragraph{\textit{Spectrum}}
The spectrum of a graph $\G$, denoted $\mathrm{Spec}(\G)$, is the spectrum of its adjacency matrix $A$ (i.e.\@  
the set of eigenvalues of $A$ counted with multiplicities).
If $\Gamma$ has different eigenvalues $\lambda_0, \ldots, \lambda_t$ with multiplicities $m_0,\ldots,m_t$, we write 
as usual 
\begin{equation*} \label{spec}
	\mathrm{Spec}(\Gamma) = \{[\lambda_0]^{m_0}, \ldots, [\lambda_t]^{m_t}\}.
\end{equation*}

It is well-known that an $n$-regular graph $\G$ has $n$ as one of its eigenvalues, with multiplicity equal to the number of connected components of $\G$. 
That is, $\G$ is connected if and only if $n$ has multiplicity $1$.
The same happens for $n$-regular digraphs, i.e.\@ those directed graphs such that any vertex has the same in-degree and out-degree equal to $n$. In this case, $\G$ is strongly connected if and only if $n$ has multiplicity $1$.

The spectrum of few families of GP-graphs are known. The graphs $\G(1,q)$ and $\G(2,q)$ with $q\equiv 1\pmod 4$ are classic being the complete graphs $K_q$ and the classic Paley graphs $P(q)$, and hence with known spectra. The spectrum of $\G(k,q)$, for $k=3,4$, was computed in \cite{PV3} in the special case $k\mid \frac{q-1}{p-1}$, i.e.\@ in the case with integral spectrum (see Section \ref{sec:4}), where $q=p^m$ for some $m$. Also, in \cite{PV1} we computed the spectrum of a subfamily of semiprimitive GP-graphs, those of the form $\G(q^\ell+1, q^m)$ with $\frac{m}{(m,\ell)}$ even.

If $\Gamma$ is an $n$-regular graph, then $n$ is the greatest eigenvalue of $\G$. 
A connected $n$-regular undirected graph is called \textit{Ramanujan} if 
	$$|\lambda| = n  \qquad \text{or} \qquad |\lambda| \le 2\sqrt{n-1}$$ 
for any eigenvalue $\lambda$ of $\G$.
Ramanujan graphs are optimal expanders. 
For background on Ramanujan graphs and expanders see for instance the excellent surveys of Ram Murty \cite{Mur}, Hoory, Linial and Wigderson \cite{HLW} and Lubotzky \cite{Lub}. There are notions of Ramanujanicity for directed graphs (digraphs). 
It is both of theoretical and practical interest to obtain families of Ramanujan (di)graphs.

\subsubsection*{Outline and results}
%\paragraph{\textbf{Outline and results}}
The paper is organized as follows. 
In Section \ref{sec:2} we study the spectrum of GP-graphs in terms of Gaussian periods. 
By using period polynomials, in Section \ref{sec:3} we give explicit computations of $\G(k,q)$ for small values of $k$ and  in Section \ref{sec:4} we characterize all integral GP-graphs. In Section \ref{sec:5} we focus on the particular case of semiprimitive GP-graphs and in Section \ref{sec:6} we characterize integral Ramanujan graphs for $\G(k,q)$ with $1 \le k \le 4$ or $(k,q)$ a semiprimitive pair. 
Sections~\ref{sec:2} and \ref{sec:5} can be thought as a kind of survey with some extra new material or with a different exposition, while the other sections present completely new results.

Let $q=p^m$ with $p$ prime, assume that $k \mid q-1$ and put $n=\frac{q-1}{k}$. We now summarize the main results of the paper.

In Section~\ref{sec:2} we study the spectrum of GP-graphs. In Theorem \ref{Spectro Gkq} we show 
that the spectrum of $\G(k,q)$ can be put in terms of the cyclotomic Gaussian periods. More precisely, 
$$\mathrm{Spec}(\G(k,q)) = \{ [n]^{1+\mu n}, [\eta_{i_1}]^{\mu_{i_1} n}, \ldots, [\eta_{i_s}]^{\mu_{i_s} n} \}$$  
where $\eta_{i_1}^{(k,q)},\ldots, \eta_{i_s}^{(k,q)}$ are the different cyclotomic Gaussian periods and the $\mu_{i_j}$'s are certain numbers (see \eqref{gaussian} and \eqref{numbers}). 

The next section is devoted to explicit computations (based on previous works of Myerson \cite{My}, Gurak \cite{Gu1}, \cite{Gu2}, and Hoshi \cite{Ho} on period polynomials). In Theorems \ref{gp3q} and \ref{gp4q} we give the whole spectrum of $\G(3,q)$ and $\G(4,q)$, respectively. This, together with Examples~\ref{g1q} and \ref{Ex:Paley} and Remark \ref{g6812q} shows that the spectrum of $\G(k,q)$ with $k\mid q-1$ can be computed for every proper divisor $k$ of $24$ (i.e.\@ $k=1,2,3,4,6,8,12)$. Moreover, we give the spectrum of $\G(5,q)$ in half of the cases: the case $p\equiv 1 \pmod 5$ is given in Proposition \ref{G5} while the case $p\equiv -1 \pmod 5$ corresponds to the semiprimitive case and hence it is obtained by taking $k=5$ in Theorem \ref{semiprimitive} (the cases $p\equiv \pm 2 \pmod 5$ remain open).

In Section \ref{sec:4} we study integrality of the spectrum by way of period polynomials. In Theorem~\ref{Teo: GP enteros}, one of the main results, we show that $\mathrm{Spec}(\G(k,q))\subset \Z$ if and only if $k\mid\frac{q-1}{p-1}$.

In Section \ref{sec:5}, we first recall the definition of semiprimitive GP-graphs and give some infinite families of these graphs.
Then, in Subsection \ref{sec:6.1} we explicitly give the spectrum of semiprimitive GP-graphs by using Gauss periods (see Theorem \ref{semiprimitive}). 
Previously, in \cite{BWX}, Brouwer, Wilson and Xiang computed the spectra of a more general family defined in terms of semiprimitive pairs by using Gauss sums. Semiprimitive GP-graphs have three different eigenvalues; hence, in the connected case, they are strongly regular graphs (and hence distance regular graphs). In Subsection \ref{sec:6.2} we give the parameters of the semiprimitive GP-graphs as strongly regular graphs, as distance regular graphs and as pseudo-Latin square graphs (see Theorem \ref{srg}).

In Section \ref{sec:6} we study some families of Ramanujan GP-graphs. First, we characterize all semiprimitive GP-graphs which are Ramanujan. In Theorem~\ref{rama car}, another main result in the paper, we prove that if $\G(k,q)$ is semiprimitive, then it is Ramanujan if and only if $k=2,3,4,5$ and $q=p^m$ satisfy certain easy arithmetic conditions. In particular, we obtain eight infinite families of semiprimitive (hence integral) Ramanujan GP-graphs $\{\G(k,p^{2t})\}_{t\in\N}$, out of which five are valid for infinite different primes $p$. 
Finally, we show that all integral GP-graphs $\G(k,q)$ with $1\le k\le 4$ which are non-semiprimitive are Ramanujan.

\section{The spectrum of GP-graphs via cyclotomic Gaussian periods} \label{sec:2}
Here, we will express the spectra of an arbitrary GP-graph $\Gamma(k,q)$, of its complement $\bar \G(k,q)$, 
and of the associated sum graph $\G^+(k,q)$, in terms of cyclotomic Gaussian periods. 
We remark that $\bar \G(k,q)$ is not in general a GP-graph (unless $k=2$), but a union of Cayley graphs, since 
	$$\bar \G(k,q) = X(\ff_q, R_k^c \smallsetminus \{0\})= X(\ff_q, C_{1}^{(k,q)}) \cup \cdots \cup X(\ff_q, C_{k-1}^{(k,q)}),$$
where 
	$$C_{i}^{(k,q)} = \omega^{i} \, \langle \omega^k \rangle$$
is the coset in $\ff_q^*$ of the subgroup $\langle \omega^{k} \rangle$ with $\omega$ a generator of $\ff_{q}^*$.
In particular, the classic Paley graphs $\G(2,q)$ with $q\equiv 1 \pmod 4$ is the only GP-graph which is self-complementary.

We begin by recalling the definition and basic properties of Gaussian periods. 
Let $p$ be a prime, take $q=p^m$ with $m \in \N$ and let $k\mid{q-1}$. 
For any $i\in \{0,1,\ldots,k-1\}$, the $i$-th \textit{cyclotomic Gaussian period} is defined by
\begin{equation} \label{gaussian}
	\eta_{i}^{(k,q)} = \sum_{x\in C_{i}^{(k,q)}} \zeta_p^{\Tr_{q/p}(x)} \: \in \mathbb{Q}(\zeta_p), 
	 \qquad \text{for $0 \le i \le k-1$,}
\end{equation} 
where $\zeta_p = e^{\frac{2\pi i}{p}}$ and $\Tr_{q/p} : \ff_q \rightarrow \ff_p$ is the trace map given by 
$$\Tr_{q/p}(x) = x+x^p+x^{p^2}+\cdots +x^{p^{r-1}}.$$

The following relation is well-known (see for instance Proposition 1 in \cite{My}):
\begin{equation} \label{sum etas=1}
	\sum_{i=0}^{k-1} \eta_i^{(k,q)} = -1.
\end{equation}	
From Theorem 13 in \cite{DY} (see also \cite{My}), if we consider 
\begin{equation} \label{N} 
	N = \gcd(\tfrac{q-1}{p-1}, k), 
\end{equation}  
we have the following integrality results:
\begin{equation} \label{int gp}
	\eta_i^{(N,q)} \in \Z \qquad \text{and} \qquad N \eta_i^{(N,q)} +1 \equiv 0 \pmod p
\end{equation}
(actually, in \cite{DY} other notations are used: $ N$ and $N_1$ for our $k$ and $N$, respectively).

\subsubsection*{The spectrum of GP$^*$-graphs}
Let $\eta_0 = \eta_0^{(k,q)}, \ldots, \eta_{k-1} = \eta_{k-1}^{(k,q)}$ 
be the cyclotomic Gaussian periods as in \eqref{gaussian} and let 
\begin{equation} \label{Gaussian periods diferentes}
	\eta_{i_1},\ldots, \eta_{i_s}
\end{equation}
denote the different cyclotomic Gaussian periods not equal to $n=\tfrac{q-1}k$. We define the following numbers  
\begin{equation} \label{numbers} 
	\begin{aligned}
		\mu 		& = \#\{0 \le \ell \le k-1 : \eta_\ell =n\} \ge 0, \\[1mm] 
		\mu_{i_j} 	& = \#\{ 0 \le \ell \le k-1 : \eta_{\ell} = \eta_{i_j}\} \ge 1,
	\end{aligned}	
\end{equation}
for $1 \le j \le s$. For simplicity, sometimes we will need to use the notation $\mu =\mu_{i_0}$.

We now show that the spectra of both GP-graphs, their complements, and GP$^+$-graphs are determined by the Gaussian periods.
We recall that $\G^+(k,q)=\G(k,q)$ for $q$ even.

\begin{thm} \label{Spectro Gkq}
Let $q=p^m$ with $p$ prime and $k \in \N$ such that $k\mid q-1$. 
If we put $n=\frac{q-1}k$ then, in the notations in \eqref{Gaussian periods diferentes} and \eqref{numbers}, we have 
\begin{equation} \label{spec Gkq} 
	\mathrm{Spec}(\G(k,q)) = \{ [n]^{1+\mu n}, [\eta_{i_1}]^{\mu_{i_1} n}, \ldots, [\eta_{i_s}]^{\mu_{i_s} n} \}
\end{equation}
and $\mathrm{Spec}(\bar \G(k,q)) = \{ [(k-1)n]^{1+\mu n}, [-1-\eta_{i_1}]^{\mu_{i_1} n}, \ldots, [-1-\eta_{i_s}]^{\mu_{i_s} n} \}$.
Furthermore, if $q$ is odd and $n$ is even then  
\begin{equation} \label{spec Gkq+} 
	\mathrm{Spec}(\G^+(k,q)) = \{ [n]^{1+\mu n}, [\pm \eta_{i_1}]^{\frac 12 \mu_{i_1} n}, \ldots, [\pm \eta_{i_s}]^{\frac 12 \mu_{i_s} n} \}.
\end{equation}
Moreover, in any case, $\G(k,q)$, $\G^+(k,q)$ and $\bar \G(k,q)$ are (strongly) connected if and only if $\mu=0$ (with $k>1$ for $\bar \G(k,q)$). 
\end{thm}

\begin{proof}
We first compute the eigenvalues of $\G(k,q)$. It is well-known that the spectrum of a Cayley graph $X(G,S)$ is determined by the irreducible characters of $G$. In fact, if $G$ is abelian, each irreducible character $\chi$ of $G$ induces an eigenvalue $\lambda_\chi$ of $X(G,S)$ by the expression 
	\begin{equation*}\label{Eigencayley}
		\lambda_\chi: = \chi(S) = \sum_{g \in S} \chi(g)
	\end{equation*} 
with eigenvector $v_{\chi} = \big( \chi(g) \big)_{g\in G}$.
	
For $\Gamma(k,q)$ we have $G=\ff_q$ and $S=R_k=\{x^k: x\in \ff_q^*\}$. 
The irreducible characters of $\ff_{q}$ are $\{ \chi_\gamma \}_{\gamma \in \ff_q}$ where 
	\begin{equation*}\label{chi gamma y}
		\chi_{\gamma}(y) = \zeta_{p}^{\Tr_{q/p}(\gamma y)} 
	\end{equation*}
for $y \in \ff_q$. Thus, since $R_{k} = \langle \omega^{k} \rangle = C_{0}^{(k,q)}$, the eigenvalues 
$\lambda_\gamma=\lambda_{\chi_\gamma}$ of $\G(k,q)$ are given by 
	\begin{equation} \label{chi gamma y2}
		\lambda_{\gamma} = \chi_{\gamma}(R_{k}) = 
		\sum_{y\in R_{k}}\chi_{\gamma}(y)=\sum_{y\in C_{0}^{(k,q)}}\zeta_{p}^{\Tr_{q/p}(\gamma y)}.
	\end{equation}
	
We have the disjoint union 
	$$\ff_q = \{0\} \cup C_0^{(k,q)} \cup \cdots \cup C_{k-1}^{(k,q)}$$ 
and $\# C_i^{(k,q)} = \# \langle \omega^k \rangle = \frac{q-1}k$ for every $i = 0, \ldots, k-1$. 
Now, for $\gamma=0$ we have 
	$$\lambda_0=\chi_0(R_k) = |R_k| = n,$$
since $\chi_0$ is the trivial character. 
This is in accordance with the fact that since $\G(k,q)$ is $n$-regular with $n=\frac{q-1}k$, then $n$ is an eigenvalue of $\G(k,q)$. 
On the other hand, if $\gamma \in C_{i}^{(k,q)}$ then $\gamma y$ runs over $C_{i}^{(k,q)}$ when $y$ runs over $C_{0}^{(k,q)}$
and thus, by \eqref{chi gamma y2}, we have 
	\begin{equation*}\label{AutTq}
		\lambda_{\gamma} = \sum_{x\in C_{i}^{(k,q)}} \zeta_{p}^{\Tr_{q/p}(x)} = \eta_{i}^{(k,q)}
	\end{equation*} 
which does not depend on $\gamma$.

Let $\eta_{i_1}, \ldots,\eta_{i_s}$ be the different cyclotomic Gaussian periods.
Notice that each $\gamma \in C_{i_\ell}^{(k,q)}$ gives the same $\lambda_\gamma$ and that $|C_{i_\ell}^{(k,q)}|=|C_0^{(k,q)}|=n$ for $1\le \ell\le s$.
Thus, it is clear that the multiplicity of $\lambda_\gamma$ is 
	$$m(\lambda_0) = 1+\sum_{\substack{0\le j \le k-1 \\[.5mm] \eta_j=n}} |C_j^{(k,q)}|$$ 
and 
	$$m(\lambda_\gamma) = \sum_{\substack{ 0\le j \le k-1 \\[0.5mm] \eta_{i_\ell} = \eta_j}} |C_j^{(k,q)}| \qquad (\text{for } \gamma \ne 0),$$	
that is $m(n)=1+\mu n$ and $m(\eta_{i_\ell}) = \mu_{i_\ell} n$  for $1 \le \ell \le s$.  This gives the spectrum for $\G(k,q)$.

To see the spectrum of the complementary graph, if $A$ is the adjacency matrix of $\G(k,q)$ then $J-A-I$ is the adjacency matrix of $\bar \G(k,q)$, where $J$ stands for the all $1$'s matrix. Since $\G(k,q)$ is $n$-regular with $q$ vertices, then $\bar \G(k,q)$ is $(q-n-1)$-regular, that is 
	$$\bar \lambda_0= q-n-1=(k-1)n.$$
The remaining eigenvalues of $\bar \G(k,q)$ are $-1-\lambda$ where $\lambda$ are the non-trivial eigenvalues, and 
hence the result follows by  \eqref{spec Gkq}.

Now consider the spectrum of $\G^+(k,q)$. Since $q$ is odd we have $\G^+(k,q) \ne \G(k,q)$ and, by Proposition 2.10 in \cite{PV5}, we get that the non-principal eigenvalues of $\G^+(k,q)$ and their corresponding multiplicities are given by 
	$$\lambda_{\G^+} = \pm \lambda_{\G} \qquad \text{and} \qquad m(\lambda_{\G^+})= \tfrac 12 m(\lambda_{\G}),$$ 
where $\lambda_{\G}$ and $m(\lambda_{\G})$ (respectively $\lambda_{\G}^+$ and $m(\lambda_{\G}^+)$) are the eigenvalues and multiplicities of $\G(k,q)$ (respectively $\G^+(k,q)$). 
Thus, we get the expression for $\mathrm{Spec}(\G^+(k,q))$ in \eqref{spec Gkq+} and we need $n$ even for the multiplicities to be integers.

Finally, being $n$-regular, $\G(k,q)$ is connected if and only if the multiplicity of $n$ is 1, i.e.\@ if $\mu=0$. A similar argument applies for the graphs $\G^+(k,q)$ and $\bar \G(k,q)$. To conclude, just notice that for $k=1$ we have $\G(1,q)=K_q$ and hence $\bar \G(1,q)$ is the empty graph with $q$ vertices which is disconnected. In this case, since $\mathrm{Spec}(K_q)=\{[q-1]^1, [-1]^{q-1}\}$ we have $\mathrm{Spec}(\bar K_q)=\{[0]^1,[-1-(-1)]^{q-1}\}=\{[0]^q\}$, although $\mu=0$ (also $-1-\eta_0=0$ by \eqref{sum etas=1}).
\end{proof}

We now make some observations on the previous theorem and point out some consequences of it for the spectrum of $\G(k,q)$.

\begin{rem} \label{Spec remarks}
$(i)$
If the Gaussian periods are all different, i.e.\@ $\eta_i \ne \eta_j$ for $0 \le i < j \le k-1$, then $\mu=0$ and $\mu_{i_j}=1$ for every $j$  since $q=kn+1$, and hence we have
	$$\mathrm{Spec}(\G(k,q)) = \{ [n]^1, [\eta_0]^n, [\eta_1]^n, \ldots, [\eta_{k-1}]^n\}.$$ 
This holds, for instance, for Paley graphs $\G(2,q)$ --both directed and undirected-- and also for the graphs $\G(3,q)$ and $\G(4,q)$ in the non-semiprimitive case  (i.e.\@ $p\not\equiv -1 \pmod k$), as one can see in Example \ref{Ex:Paley} and Theorems~\ref{gp3q} and \ref{gp4q}, respectively. 

\noindent $(ii)$
Notice that, by the theorem, $\G(k,q)$ is integral if and only if $\bar \G(k,q)$ and $\G^+(k,q)$ are integral. 
In Section \ref{sec:4} we will characterize all integral GP-graphs (and hence all integral complements and all integral GP$^+$-graphs).

\noindent $(iii)$
Theorem \ref{Spectro Gkq} allows one to compute the spectrum of families of GP-graphs in those cases where the (cyclotomic) Gaussian periods or Gaussian sums are known. 
The Gaussian periods $\eta_i^{(k,q)}$ for $k=2,3,4,6,8,12$ are well-known; the cases $k=2,3,4$ date back to Gauss (see for instance \cite{My}) while the cases $k=6,8,12$ are due to Gurak (\cite{Gu1}, \cite{Gu2}). The case $k=5$ is partially done by Hoshi \cite{Ho}. A case which is well understood is when $(k,q)$ is a semiprimitive pair. 
Some of these cases will be treated in more detail in Sections \ref{sec:3} and \ref{sec:5}. 
Other general examples of known Gaussian sums are the so-called index 2 and index 4 cases (see the literature).
It would be interesting to find the spectrum of $\G(k,q)$ in these cases. 
\end{rem}

To close the section we illustrate with two basic examples. 
We compute the spectrum of $\G(k,q)$ for $k=1,2$. Using Theorem \ref{Spectro Gkq} one can also obtain the spectrum of $\bar \G(k,q)$ and $\G^+(k,q)$ for $k=1,2$ (we leave the details).
More involved computations will be performed in the next section.

\begin{exam}[\textit{Complete graphs}] \label{g1q}
We have $\G(1,q)=K_q$ and $\mathrm{Spec}(K_q)=\{[q-1]^1,[-1]^{q-1}\}$.
Using Theorem \ref{Spectro Gkq}, since $n=q-1$ and $\mu=0$, $\mu_1=1$ by \eqref{numbers}, we obtain that
	$$\mathrm{Spec}(\G(1,q)) = \{[q-1]^1,[\eta_0]^{q-1}\},$$ 
and $\eta_0=-1$ by \eqref{sum etas=1}, hence recovering the known result. \hfill $\lozenge$
\end{exam}

\goodbreak 

\begin{exam}[\textit{Paley graphs}] \label{Ex:Paley}	
We recall that $\G(2,q)$ with $q=p^m$ is the classic (undirected) Paley graph $P(q)$ if $q\equiv 1 \pmod 4$, hence $p\equiv 1 \pmod 4$ or $p\equiv 3 \pmod 4$ and $m=2t$; and it is the directed Paley graph $\vec{P}(q)$ if $q\equiv 3 \pmod 4$, hence $p\equiv 3 \pmod 4$ and $m=2t+1$.

If we put $n=\frac{q-1}2$, by Theorem \ref{Spectro Gkq} we have that 
	$$\mathrm{Spec}(\Gamma(2,q)) = \{ [n]^{1+\mu}, [\eta_0]^{1+\mu_0 n}, [\eta_1]^{1+\mu_1 n}\}$$
where $\eta_i=\eta_i^{(2,q)}$ for $i=0,1$. The above Gaussian periods are known, see for instance Lemma~11 in \cite{DY}. In our notations, (i.e.\@ taking $s=1$ in \cite{DY}, $r=p^m$ is our $q$) we have  
\begin{equation*} \label{eta_0 (2,q)}
	\eta_0= \begin{cases}
	\tfrac 12 \big(-1+(-1)^{m-1} \sqrt q \big) 				& \qquad \text{if $p\equiv 1 \pmod 4$}, \\[1.5mm]  
	\tfrac 12 \big(-1+(-1)^{m-1} \sqrt{-1}^m \sqrt q \big)   & \qquad \text{if $p\equiv 3 \pmod 4$},
	\end{cases}	
\end{equation*}
and $\eta_1=-1-\eta_0$. First, notice that $\eta_0$ and hence $\eta_1$ are real if and only if $\G(2,q)$ is undirected.
Second, note that $\eta_0 \ne \eta_1$ and that $n \ne \eta_0,\eta_1$. These last conditions imply that $\mu=0$ and $\mu_0=\mu_1=1$ (or conversely, since $2n+1=q$ we must have that $\mu=0$ and $\mu_0=\mu_1=1$). 
Hence, we have that  
	$$\mathrm{Spec}(\Gamma(2,q)) = \{ [n]^{1}, [\eta_0]^{n}, [\eta_1]^{n}\},$$ 
where
\begin{align*}
	& \eta_0 = \begin{cases} 
	\hfil \frac{-1-p^t}2   & \qquad \qquad \text{if $p\equiv 1 \pmod 4$}, \\[1.5mm]
	\tfrac{-1-(-1)^t p^t}2 & \qquad \qquad \text{if $p\equiv 3 \pmod 4$},
	\end{cases} \qquad \qquad \text{for $m=2t$},  \\[2mm]
	& \eta_0 = \begin{cases} 
	\hfil \frac{-1+p^t\sqrt{p}}2   & \quad \qquad \text{if $p\equiv 1 \pmod 4$}, \\[1.5mm]
	\tfrac{-1-(-1)^t i p^t\sqrt{p}}2 & \quad \qquad \text{if $p\equiv 3 \pmod 4$},
	\end{cases} \qquad \qquad \text{for $m=2t+1$},
\end{align*} 
and $\eta_1=-1-\eta_0$. Notice that $\eta_0 \in \Z$ for $m=2t$ and $\eta_0 \in \mathbb{Q}(\sqrt{-p})$ for $m=2t+1$. 
This coincides with the known spectrum 
	\begin{equation*} \label{spec paley}
		\mathrm{Spec}(P(q)) = \big \{ [\tfrac{q-1}2]^1, [\tfrac{-1+(-1)^m \sqrt{q}}2]^{n}, [\tfrac{-1-(-1)^m\sqrt{q}}2]^{n} \big \}.
	\end{equation*}
We also obtain
	\begin{equation} \label{spec paley dir}
		\mathrm{Spec}(\vec{P}(q)) = \big\{ [\tfrac{q-1}2]^1, [\tfrac{-1+(-1)^m i^m \sqrt{q}}2]^{n}, [\tfrac{-1-(-1)^m i^m \sqrt{q}}2]^{n} \big\}.
	\end{equation}

Finally, the graph $\G(2,q)$ is connected since the multiplicity of the regularity degree is 1. 
Therefore, since $\G(2,q)$ is regular and connected and has exactly 3 eigenvalues it is a strongly regular graph (in the undirected case). All these facts are of course well-known.  
\hfill $\lozenge$
\end{exam}

%\blue 
\section{Explicit computations through period polynomials} \label{sec:3}
Since the spectrum of GP-graphs is given in terms of Gaussian periods, we now recall the polynomial associated with them. The \textit{period polynomial} is defined by 
\begin{equation} \label{period polynomial}
	\Psi_{k,q}(x) = \prod_{i=0}^{k-1} (x-\eta_i^{(k,q)}), 
\end{equation} 
where $\eta_i^{(k,q)}$ is the Gaussian period given in \eqref{gaussian}.

In some cases, the expansions of these polynomials are known, and they factor into product of polynomials of small degree, hence their roots (the Gaussian periods) can be explicitly computed.  
We will use some of the known cases to give the spectrum of the associated GP-graphs explicitly. In particular, the spectrum of the graphs $\G(k,q)$ with $k\mid 24$ ($k\ne 24$) can be determined (although it is highly non-trivial in most of the cases). For simplicity, we will give explicitly the spectrum of the graphs $\G(k,q)$ with $k=3,4$ (the cases $k=1,2$ were presented in Examples \ref{g1q} and \ref{Ex:Paley}). For the remaining cases $k=6,8,12$ we refer to the works of Gurak. Using results of Hoshi we give the spectrum of $\G(5,p^{5t})$ in the case $p\equiv 1 \pmod 5$. 
Another well-known case is the semiprimitive one, which is delayed until Section~\ref{sec:5}.

In \cite{PV3}, we have computed the spectrum of $\G(k,q)$ for $k\mid \frac{q-1}{p-1}$ with $k=3,4$, where $q=p^m$. 
There, we used a relation that we found between the spectrum of GP-graphs $\G(k,q)$ and the weight distribution of certain irreducible cyclic codes $\mathcal{C}(k,q)$, provided that $k\mid \frac{q-1}{p-1}$ (a posteriori, those GP-graphs having integral spectrum, see Section \ref{sec:4}). Namely, we have used this relation and the fact that the weight distributions for the codes $\CC(3,q)$ and $\CC(4,q)$ was already known (which was computed by using Gaussian periods) in these cases. 

Now, we will give the complete result, that is we give the spectrum of $\G(k,q)$ for $k\mid q-1$ with $k=3,4$, by way of explicit factorizations of the period polynomials $\Psi_{3,q}(x)$ and $\Psi_{4,q}(x)$ obtained by Myerson (\cite{My}). It turns out that there is one extra case for $k=3$ and two extra cases for $k=4$ in this more general setting (i.e. $k\mid q-1$ instead of $k\mid \frac{q-1}{p-1}$).

We now give the spectrum of the GP-graphs $\G(3,q)$ explicitly. 
\begin{thm}	\label{gp3q}
Let $q=p^{m} \ge 5$ with $p$ prime such that $3\mid q-1$ and put $n=\frac{q-1}3$. Thus, the graph $\G(3,q)$ 
is connected and undirected with real spectrum given as follows:
\begin{enumerate}[$(a)$]
	\item If $p\equiv 1 \pmod 3$ with $3\mid m$ then 
		$$\mathrm{Spec}(\G(3,q)) = \big\{ [n]^1, \big[\tfrac{a\sqrt[3]{q}-1}{3}\big]^n, \big[\tfrac{-\frac{1}{2} (a+9b)\sqrt[3]{q}-1}{3}\big]^n, 
		\big[\tfrac{-\frac 12 (a-9b) \sqrt[3]{q}-1}{3}\big]^n \big\} $$
	where $a,b$ are integers uniquely determined by 
		\begin{equation} \label{ab27}
			4\sqrt[3]{q}=a^2+27b^2, \qquad a\equiv 1 \pmod 3 \qquad \text{and} \qquad (a,p)=1.
		\end{equation} 
		
	\item  
		If $p\equiv 1 \pmod 3$ with $3 \nmid m$ then $\mathrm{Spec}(\G(3,q)) = \{ [n]^1, [x_0]^n, [x_1]^n, [x_2]^n \}$ 
		where 
			$$x_j= -\tfrac 13 \big( 1+\omega^j W + \frac{q}{\omega^j W} \big),	\qquad j\in \{0,1,2\},$$ 
		$\omega=e^{\frac{2\pi i}{3}}$ and 
			$W = \sqrt[3]{q} \sqrt[3]{\tfrac 12 (-a + \sqrt{-27} b)}$,	
		where $a$ and $b$ are uniquely determined integers ($b$ up to sign) satisfying 
		$$4q=a^2+27b^2, \qquad a\equiv 1 \pmod 3 \qquad \text{and} \qquad (a,p)=1$$
		Here, $\sqrt{}$ and $\sqrt[3]{}$ denote any square and cubic root, respectively. \msk 
	
	\item If $p\equiv 2 \pmod 3$ with $m$ even then  
		$$\mathrm{Spec}(\G(3,q)) = \begin{cases} 
		\big\{ [n]^1, \big[\tfrac{\sqrt{q}-1}{3}\big]^{2n}, \big[\tfrac{-2\sqrt{q}-1}{3}\big]^n \big\} & 
		\qquad \text{for $m\equiv 0 \pmod 4$}, \\[3mm]
		\big\{ [n]^1, \big[\tfrac{2\sqrt{q}-1}{3}\big]^{n}, \big[\tfrac{-\sqrt{q}-1}{3}\big]^{2n} \big\} & 
		\qquad \text{for $m\equiv 2 \pmod 4$}.
		\end{cases}$$
\end{enumerate}

Furthermore, the spectrum is integral in cases $(a)$ and $(c)$. 
\end{thm}

\begin{proof}
The graph $\G(3,q)$ is connected and undirected (see the Introduction) since $\frac{q-1}3$ is a primitive divisor of $q-1$ and for $q$ odd we have that $3 \mid \frac{q-1}2$ (in fact $2\mid q-1$ and $3\mid q-1$, hence $6\mid q-1$, which is equivalent to $3\mid \frac{q-1}2$). Thus, $\mathrm{Spec}(\G(3,q)) \subset \mathbb{R}$  (the adjacency matrix of an undirected graph is symmetric and so its spectrum is real).
	
In Theorems 13 and 16 in \cite{My} (see also Lemmas 7 and 8 in \cite{DY}) Myerson gave the polynomial $\Psi_{3,q}(x)$ and its factorizations over the rationals. Namely, 
	\begin{equation} \label{Psi3x}
		\Psi_{3,q}(x) = x^3+x^2-nx-d \qquad \text{with} \qquad d =\tfrac{(a+3)q-1}{27},
	\end{equation}	 
where $a$ and $b$ are integers uniquely determined ($b$ only up to sign) by 
$4q=a^2+27b^2$,
with $a \equiv 1\pmod 3$ and if $p\equiv 1 \pmod 3$ then $(a,p)=1$. 
We have the following cases: \vspace{-1mm}
\begin{enumerate}[$(a)$]
	\item If $p\equiv 1 \pmod 3$ and $3 \mid m$, then 
		$$\Psi_{3,q}(x)= \tfrac{1}{27} (3x+1-a \sqrt[3]{q}) (3x+1-a \sqrt[3]{q})(3x+1-a \sqrt[3]{q}),$$ 
 	where $a$ and $b$ are as in \eqref{ab27}. \sk 
 	
	\item If $p\equiv 1 \pmod 3$ and $3 \nmid m$, then $\Psi_{3,q}(x)$ is irreducible over $\mathbb{Q}$. \sk 
	
	\item If $p\equiv 2 \pmod 3$ and $m$ is even, then 
		$$\Psi_{3,q}(x)= \begin{cases}
		\tfrac{1}{27} (3x+1+2\sqrt{q}) (3x+1-\sqrt{q})^2 & \qquad \text{if $\frac m2$ is even}, \\[1.5mm]
		\tfrac{1}{27} (3x+1-2\sqrt{q}) (3x+1+\sqrt{q})^2 & \qquad \text{if $\frac m2$ is odd}. 
		\end{cases}$$
\end{enumerate}

Thus, the eigenvalues and multiplicities of $\G(3,q)$ are directly obtained from these expressions in cases $(a)$ and $(c)$. 

Now, we study case $(b)$. We have to find the roots of $\Psi_{3,q}(x)$ in \eqref{Psi3x}. 
The roots of a general cubic $Ax^3+Bx^2+Cx+D$ are given by 
	$$x_j = -\tfrac{1}{3A} \big( B +\omega^j W + \tfrac{\Delta_0}{\omega^j W} \big)$$ 
for $j=0,1,2$, where $\omega=e^{\frac{2\pi i}{3}}$ is the primitive cubic root of 1, and 
	$$W = \sqrt[3]{\tfrac 12 \big( \Delta_1 \pm \sqrt{\Delta_1^2-4\Delta_0^3} \big)}$$
where $\Delta_0=B^2-3AC$ and $\Delta_1 =2B^3-9ABC+27A^2D$.
Thus, by \eqref{Psi3x}, we have that $A=B=1$ and we obtain that $\Delta_0=q$ and $\Delta_1 =-aq$. In this way, we arrive at
	$$W= \sqrt[3]{\tfrac 12 \big( -aq \pm q\sqrt{a^2-4q} \big)} = \sqrt[3]{q} \sqrt[3]{\tfrac 12(-a \pm \sqrt{-27}b)},$$
where we have used that $4q = a^2 + 27 b^2$. 

Here, $\sqrt{}$ and $\sqrt[3]{}$ denote any square and any cubic root, respectively. In general, the sign $\pm$ can be randomly chosen, and if $W=0$ with one sign one has to chose the other one. 
It can never happen that both signs give $W=0$, since this is equivalent to $\Delta_1=\Delta_0=0$ and both $\Delta_0$ and $\Delta_1$ are non-zero in our case. Hence, we choose the plus sign, and part $(b)$ is proved.
	
Finally, the eigenvalues in cases $(a)$ and $(c)$ are integers by the conditions on $p$ and $m$. In case ($a$) we have that $a\equiv a\pm 9b \equiv 1 \pmod 3$ where $a\pm 9b$ is even since $4p^m=a^2+27b^2$ implies that $a$ and $b$ have the same parity. The remaining assertion is clear from the statement. 
\end{proof}

 Using the theorem one can compute, for instance, $\mathrm{Spec}(\G(3,7^{3m}))$ with item $(a)$, $\mathrm{Spec}(\G(3,7^{3m+j}))$, $j=1,2$, with item $(b)$ and $\mathrm{Spec}(\G(3,5^{2m}))$ with item $(c)$ for any $m\in \N$.

Next we give the spectrum of the GP-graphs $\G(4,q)$ explicitly. 
\begin{thm}	\label{gp4q}
Let $q=p^m$ with $p$ prime such that $4 \mid q -1$ %with %$q\ne 9$ 
and put $n=\frac{q-1}4$. 
Thus, the graph $\G(4,q)$ is connected (except for $q=9$) %and undirected 
with spectrum given as follows:
	\begin{enumerate}[$(a)$]
		\item If $p\equiv 1 \pmod 4$ with $m \equiv 0 \pmod 4$ then  
		$$\mathrm{Spec}(\G(4,q)) =  \Big\{ [n]^1, \big[\tfrac{\sqrt{q} + 4d\sqrt[4]{q}-1}{4}\big]^n, \big[\tfrac{\sqrt{q} - 4d\sqrt[4]{q}-1}{4}\big]^n, \big[\tfrac{-\sqrt{q} + 2c \sqrt[4]{q}-1}{4}\big]^n, \big[\tfrac{-\sqrt{q} - 2c \sqrt[4]{q}-1}{4}\big]^n \Big\}$$
		where $c,d$ are integers uniquely determined by 
		\begin{equation} \label{c4d}
			\sqrt{q}=c^2+4d^2, \qquad c\equiv 1 \pmod 4 \qquad \text{and} \qquad (c,p)=1.
		\end{equation} 
		
		\item If $p\equiv 1 \pmod 4$ with $m \equiv 2 \pmod 4$, then 
		\begin{multline*} \mathrm{Spec}(\G(4,q)) = \Big\{ [n]^1, \big[\tfrac{-(1+\sqrt q) + \sqrt{2(q+c\sqrt{q})}}{4}\big]^n, 
		\big[\tfrac{-(1+\sqrt q) - \sqrt{2(q+c\sqrt{q})}}{4}\big]^n, \\	\big[\tfrac{-(1-\sqrt q) + \sqrt{2(q-c\sqrt{q})}}{4}\big]^n,  \big[\tfrac{-(1-\sqrt q) - \sqrt{2(q-c\sqrt{q})}}{4}\big]^n \Big\}
		\end{multline*}
		where $c,d$ are integers unique determined ($d$ up to sign) by  
	   \begin{equation} \label{c04d0}
			q=c^2+4d^2, \qquad c \equiv 1 \pmod 4 \qquad \text{and} \qquad (c,p)=1.
	  \end{equation}

		\item If $p\equiv 1 \pmod 4$ with $m$ odd and $n$ odd then 
			$$\mathrm{Spec}(\G(4,q)) = \{ [n]^1, [x_1^+]^n, [x_1^-]^n, [x_2^+]^n, [x_2^-]^n \}$$ 
		with 
\begin{align*} 
\begin{split}
&	x_1^\pm = \tfrac 12 \Big( -\sqrt{2y-\tfrac q8} \pm \sqrt{-(2y + \tfrac q8) + \tfrac{(c-1)q+1}{4\sqrt{2y-\tfrac q8}}}
\Big) -\tfrac 14, \\[1.5mm]
&	x_2^\pm = \tfrac 12 \Big( \sqrt{2y-\tfrac q8} \pm \sqrt{-(2y + \tfrac q8) - \tfrac{(c-1)q+1}{4\sqrt{2y-\tfrac q8}}}
\Big) -\tfrac 14,
\end{split}
\end{align*} 
where $y=\frac{q}{48}+ W -\frac{P}{3W}$ and $W = \sqrt[3]{-\frac Q2 \pm \sqrt{\frac{Q^2}{4}+\frac{P^3}{27}}}$
 with $P= \frac{36- (28q -12c^2 -12 )q}{3 \cdot 256}$, $Q= -\frac{q^2}{27 \cdot 256} +\frac{q\gamma}{24} - \frac{((c-1)q+1)^2}{64 \gamma}$ and $\gamma = \frac{(9q-4c^2-4)q-12}{256}$, where $c,d$ are integers as in \eqref{c04d0}. 
\msk 
		\item If $p\equiv 1 \pmod 4$ with $m$ odd and $n$ even then 
				$$\mathrm{Spec}(\G(4,q)) = \{ [n]^1, [x_1^+]^n, [x_1^-]^n, [x_2^+]^n, [x_2^-]^n \}$$ 
		with 
\begin{align*} 
	\begin{split}
		&	x_1^\pm = \tfrac 12 \Big( -\sqrt{2y+\tfrac{3q}{8}} \pm \sqrt{-(2y - \tfrac{3q}{8}) + \tfrac{(3-c)q-1}{4\sqrt{2y+\tfrac{3q}{8}}}}
		\Big) -\tfrac 14, \\[1.5mm]
		&	x_2^\pm = \tfrac 12 \Big( \sqrt{2y+\tfrac{3q}{8}} \pm \sqrt{-(2y - \tfrac{3q}{8}) - \tfrac{(3-c)q-1}{4\sqrt{2y+\tfrac{3q}{8}}}}
		\Big) -\tfrac 14,
	\end{split}
\end{align*} 
where $y=-\frac{q}{16}+\omega -\frac{P}{3\omega}$ and $\omega = \sqrt[3]{-\frac Q2 \pm \sqrt{\frac{Q^2}{4}+\frac{P^3}{27}}}$
with  $P= \frac{(2q +12-4c^2)q-12}{256}$, \break 
	$Q= -\frac{q^2}{3 \cdot 256} -\frac{q\gamma}{8} - \frac{((3-c)q-1)^2}{64 \gamma}$ and $\gamma =\frac{(q+12-4c^2)q-12}{256}$, where $c,d$ are integers as in \eqref{c04d0}. 

		\item If $p\equiv 3 \pmod 4$ with $m$ even then $\mathrm{Spec}(\G(4,9))=\{[2]^3, [-1]^6\}$ and for any $q\ne 9$ we have
		$$\mathrm{Spec}(\G(4,q)) = \begin{cases} 
		\big\{ [n]^1, \big[\tfrac{\sqrt{q}-1}{4}\big]^{3n}, \big[\tfrac{-3\sqrt{q}-1}{4}\big]^n \big \} & 
		\qquad \text{for $m\equiv 0 \pmod 4$}, \\[3mm]
		\big\{ [n]^1, \big[\tfrac{3\sqrt{q}-1}{4}\big]^{n}, \big[\tfrac{-\sqrt{q}-1}{4}\big]^{3n} \big\} & 
		\qquad \text{for $m\equiv 2 \pmod 4$}.
		\end{cases}$$
	\end{enumerate}

Moreover, the graph $\G(4,q)$ is undirected with real spectrum in cases $(a)$, $(b)$, $(d)$ and $(e)$. 
In particular, the spectrum is integral in cases $(a)$ and $(e)$.
\end{thm}

\begin{proof}
The graph $\G(4,q)$ is connected (except for $q=9$) since $\frac{q-1}4$ is a primitive divisor of $q-1$ (see the Introduction). For $p$ odd and $m=2t$ even, i.e.\@ in cases ($a$), ($b$) and ($e$), one can show that $4\mid \frac{q-1}2$, and hence the graph is undirected. In fact, if $p = 4t+a$, with $a=1$ or $3$, then $p^2 \equiv a^2 \equiv 1 \pmod 8$. Thus, $p^{2s}\equiv 1 \pmod 8$ for every $s \in\N$ and $q\equiv 1 \pmod 8$. In the remaining case ($d$), the graph $\G(4,q)$ is undirected since $n=\frac{q-1}{4}$ is even. Indeed, if $2\mid \frac{q-1}{4}$, then $\frac{q-1}{4}=2t$ for some $t\in \mathbb{Z}$ which implies that $\frac{q-1}{2}=4t$ and therefore $4\mid \frac{q-1}{2}$.
Thus, $\mathrm{Spec}(\G(4,q)) \subset \mathbb{R}$  (the adjacency matrix of an undirected graph is symmetric).

Also, the graph $\G(4,9)$ is the disjoint union of three copies of $K_3$, and since $\mathrm{Spec}(K_3)=\{[2]^1, [-1]^2\}$ we get that the spectrum of $\G(4,9)$ is as stated (notice that it is still given by the corresponding formula in $(b)$, i.e.\@ $\mathrm{Spec}(\G(4,q)) = \{[n]^1, [\tfrac 14 (3\sqrt{q}-1)]^{n}, [\tfrac 14 (-\sqrt{q}-1)]^{3n} \}$ for $q=9$) since, as multisets, $\{[2]^1, [2]^2, [-1]^6\}=\{[2]^3, [-1]^6\}$.

In Theorems 14 and 17 in \cite{My} (see also Lemmas 9 and 10 in \cite{DY}) Myerson gave the polynomial $\Psi_{4,q}(x)$ and its factorizations over the rationals. 
Namely, we have 
\begin{equation} \label{Psi4q}
	\Psi_{4,q}(x) = \begin{cases}
	x^4 + x^3 - \tfrac{3q-3}8 x^2 + \tfrac{(2c-3)q+1}{16} x + \tfrac{q^2-(4c^2-8c+6)q+1}{256}  & \qquad \text{if $n$ is even}, \\[2.5mm]
	x^4 + x^3 + \tfrac{q+3}8 x^2  + \tfrac{(2c+1)q+1}{16} x + \tfrac{9q^2-(4c^2-8c-2)q+1}{256} & \qquad \text{if $n$ is odd}, \end{cases}
\end{equation}
where $c,d$ are integers uniquely determined ($d$ up to sign) such that 
$q=c^2+4d^2$, $d\equiv 1\pmod 4$ and if $p\equiv 1\pmod 4$ then $(c,p)=1$.
We have the following cases:
\begin{enumerate}[$(a)$]
	\item If $p\equiv 1 \pmod 4$ and $m\equiv 0 \pmod 4$ then $\Psi_{4,q}(x)$ equals
	{\small
	$$\tfrac{1}{64} \big((4x+1)+\sqrt{q}+2c\sqrt[4]{q} \big) \big( (4x+1)+\sqrt{q}-2c\sqrt[4]{q} \big)\big( (4x+1)-\sqrt{q}+2c\sqrt[4]{q} \big)\big( (4x+1)-\sqrt{q}-2c\sqrt[4]{q} \big),$$
where $c,d$ are integers uniquely determined by \eqref{c4d}.} \sk 
	
	\item If $p\equiv 1 \pmod 4$ and $m\equiv 2 \pmod 4$ then $\Psi_{4,q}(x)$ equals
	 $$%\Psi_{(4,q)}(x)=
	 \tfrac{1}{64} \big((4x+1)^2 + 2\sqrt{q}(4x+1) -q - 2c\sqrt{q} \big) 
	 			     \big((4x+1)^2 - 2\sqrt{q}(4x+1) -q + 2c\sqrt{q} \big)$$
	 where the quadratics are irreducible over $\mathbb{Q}$ and $c,d$ are integers satisfying 
	 $q=c^2+4d^2$, $d \equiv 1\pmod 4$ and $(c,p)=1$. \sk 
	 
	\item If $p\equiv 1 \pmod 4$ and $m$ odd then $\Psi_{(4,q)}(x)$ is irreducible over $\mathbb{Q}$. \sk 
	
	\item If $p\equiv 3 \pmod 4$ and $m$ even then
	$$\Psi_{4,q}(x)= \begin{cases}
	\tfrac{1}{64} (4x+1+3\sqrt{q}) (4x+1-\sqrt{q})^3 & \qquad \text{if $\frac m2$ is even}, \\[1.5mm]
	\tfrac{1}{64} (4x+1-3\sqrt{q}) (4x+1+\sqrt{q})^3 & \qquad \text{if $\frac m2$ is odd}. 
	\end{cases}$$
\end{enumerate}

Thus, the eigenvalues and multiplicities of $\G(4,q)$ are directly obtained from these expressions in cases $(a)$ and $(d)$. 
In case $(b)$, a routine calculation shows that the roots of these quadratics are respectively given by 
	$$\tfrac{-(1+\sqrt q) \pm \sqrt{2(q+c\sqrt{q})}}{4} \qquad \text{and} \qquad \tfrac{-(1-\sqrt q) \pm \sqrt{2(q-c\sqrt{q})}}{4},$$
from which the spectrum in this case readily follows.
	
We now consider cases $(c)$ and ($d$). The roots of a general quartic 
	$$Ax^4+Bx^2+Cx^2+Dx+E$$ 
are given by
\begin{align} \label{x_i's}
	\begin{split}
		&	x_1^\pm = \tfrac 12 \Big( -\sqrt{2y-\alpha} \pm \sqrt{-(2y + \alpha) + \tfrac{2\beta}{\sqrt{2y-\alpha}}}
		\Big) -\tfrac 14, \\[1.5mm]
		&	x_2^\pm = \tfrac 12 \Big( \sqrt{2y-\alpha} \pm \sqrt{-(2y + \alpha) - \tfrac{2\beta}{\sqrt{2y-\alpha}}}
		\Big) -\tfrac 14,
	\end{split}
\end{align}
where 
	\begin{equation} \label{y&W}
		y = \tfrac{\alpha}6 + W - \tfrac{P}{3 W} \qquad \text{and} \qquad 
		W = \sqrt[3]{-\tfrac Q2 \pm \sqrt{\tfrac{Q^2}{4} + \tfrac{P^3}{27}}}
	\end{equation} 
with $P=-\frac{\alpha^2}{12}-\gamma$ and $Q= -\frac{\alpha^2}{108} + \frac{\alpha\gamma}{3} - \frac{\beta^2}{\gamma}$
where
	$$\alpha = -\tfrac{3B^2}{8A^2} + \tfrac CA, \quad \beta  =  \tfrac{B^3}{8A^3} -\tfrac{BC}{2A^2} -\tfrac{D}{A}, 
	\quad \text{and} \quad 
	\gamma = -\tfrac{3B^4}{256A^4} - \tfrac{B^2C}{16A^3} -\tfrac{BD}{4A^2}+ \tfrac EA.$$

	If $n$ is odd, by using the second line in \eqref{Psi4q}, one can check that 
\begin{gather*}
	\alpha = \tfrac q8, \qquad \beta  =  \tfrac{(c-1)q+1}8, \qquad \gamma = \tfrac{(9q-4c^2-4)q-12}{256}, \\
		 P = \tfrac{36- (28q -12c^2 -12 )q}{256 \cdot 3}, 	\qquad \text{and} \qquad 
		 Q = -\tfrac{q^2}{27 \cdot 256} + \tfrac{q\gamma}{24} - \tfrac{((c-1)q+1)^2}{64 \gamma};
\end{gather*}
while if $n$ is even, by the first line in \eqref{Psi4q}, one has that 
\begin{gather*}
\alpha = \tfrac {-3q}{8}, \qquad \beta  =  \tfrac{(3-c)q-1}8, \qquad \gamma = \tfrac{(q+12-4c^2)q-12}{256}, \\ 
	 P = \tfrac{(2q +12-4c^2)q-12}{256}, 
	 \qquad \text{and} \qquad 
     Q = -\tfrac{q^2}{3 \cdot 256} - \tfrac{q\gamma}{8} - \tfrac{((3-c)q-1)^2}{64 \gamma}.
\end{gather*}
Putting this information in \eqref{x_i's} and \eqref{y&W} we get the desired result.

Finally, the spectra in cases ($a$) and $(e)$ are integral by the conditions on $p$ and $m$, where in case ($a$) we use that $c\equiv 1 \pmod 4$). 
%The remaining assertions are clear from the statement. 
\end{proof}

Using the theorem one can compute, for instance, the following spectra: $\mathrm{Spec}(\G(4,5^{4m}))$ with item $(a)$, $\mathrm{Spec}(\G(4,5^{4m+2}))$ with item $(b)$, $\mathrm{Spec}(\G(4,5^{2m+1}))$ with items $(c)$ and $(d)$ and $\mathrm{Spec}(\G(4,7^{2m}))$ with item $(e)$, for any $m\in \N$.

\begin{rem} \label{rem g3g4}
($i$) Theorems \ref{gp3q} and \ref{gp4q} extend the results obtained in Theorems~2.2 and 2.4 in \cite{PV3}, valid for $k\mid \frac{q-1}{p-1}$, to the general case $k\mid q-1$. As we will see in the next section, the cases considered in \cite{PV3} correspond to those with integral spectrum, and this explains why we were able to obtain them via weight distribution of codes. 
The general case (i.e.\@ with non-integral spectrum), however, cannot be related with weight distribution of codes.

\noindent 
($ii$) 
Note that the spectra of $\G(3,q)$ in Theorem \ref{gp3q} ($b$) and of $\G(4,q)$ in Theorem \ref{gp4q} ($d$) are real, although this may not look so from the expressions. For instance, in the case ($b$) of Theorem \ref{gp3q}, notice that $|W|^{2}=q$ and hence 
$\frac{q}{\omega^j W} =\overline{\omega^{j} W}$. This implies that 
	$$\omega^{j}W+\tfrac{q}{\omega^{j}W}=2\, \mathrm{Re} (\omega^{j} W)$$ 
for any $0\le j \le 2$ and therefore $\G(3,q)$ is real. However, to check that $\G(4,q)$ in case $(d)$ has real spectrum directly may be quite difficult.
\end{rem}

Now, using a result of Hoshi \cite{Ho} we can give the spectrum of $\G(5,q)$ in the case $q=p^m$ with $p\equiv 1 \pmod 5$ and $5\mid m$. 
In Theorem 1 in \cite{Ho}, for $p\equiv 1 \pmod 5$ and $m=5s$, Hoshi obtained the factorization of 
the period polynomial in the reduced form 	
	$$\Psi_{5,p^{m}}^*(X)= \prod_{i=1}^5 (X-\eta_i^*),$$ 
where $\eta_i^*=5\eta_i+1$, in terms of solutions of the so-called Dickson's system of Diophantine equations:
\begin{equation} \label{dickson}
	\begin{cases}
		16p^{m}=x^2+125w^2+50v^2+50u^2, \\
		xw=v^2-4uv-u^2, \\
		x\equiv -1 \pmod 5.
	\end{cases}
\end{equation} 
If we denote by $S(p,m)$ the set of all integer solutions of this system, it is known that $\#S(p,m)=(m+1)^2$. Moreover, the system have exactly four integer solutions satisfying $p\nmid x^2-125w^2$ and the set of these solutions is denoted by $S(p,m)^U$.

\begin{prop} \label{G5}
Let $q=p^{5s}$ with $s\in \N$, $p$ prime of the form $p\equiv 1 \pmod 5$, and put $n=\frac{q-1}5$. Then, we have
	$$\mathrm{Spec}(\G(5,q)) = \big\{ [n]^1, [\tfrac{\eta_0^*-1}5]^n, [\tfrac{\eta_1^*-1}5]^n, [\tfrac{\eta_2^*-1}5]^n, [\tfrac{\eta_3^*-1}5]^n, [\tfrac{\eta_4^*-1}5]^n \big\}$$
where 
	$$\eta_0^* = -\tfrac{1}{16}p^s(x^3-25L)%$ and % 
	\qquad \text{and} \qquad 
	\eta_i^* = \tfrac{1}{64}p^s (x^3-25M) \sigma^i \quad (0 \le i \le 3),$$
with $\sigma$ the non-singular linear transformation of order 4 given by $\sigma(x,w,v,u)=(x,-w,-u,v)$ and 
	\begin{equation*}
	\begin{aligned} 
	L = & 2x(v^2+u^2) +5w(11v^2-4vu-11u^2), \\
	M = & 2x^2u+7xv^2+20xvu-3xu^2+125w^3 + 200w^2v \\ 
	& -150w^2u + 5wv^2-20wvu -105wu^2 -40v^3-60v^2u+120vu^2+20u^3,	
	\end{aligned}			
	\end{equation*}
	for $(x,w,v,u)$ any solution of \eqref{dickson} such that $p\nmid x^2-125w^2$.
\end{prop}

\begin{proof}
	It follows by a direct application of Theorem \ref{Spectro Gkq} and Theorem 1 in \cite{Ho}.
\end{proof}

Relative to the spectrum of $\G(5,p^m)$, the case $p\equiv 4 \pmod 5$, i.e.\@ $p\equiv -1 \pmod 5$, corresponds to the semiprimitive one, and hence it is obtained by taking $k=5$ in Theorem~\ref{semiprimitive} ahead, while the cases $p\equiv 2,3 \pmod 5$ remain open in general since the Gaussian periods in these cases are unknown (except in the semiprimitive case), to our best knowledge. 

\begin{exam} \label{Ex:G5}
Here we give the spectrum of $\G(5,11^5)$.
By the last Example in Section 5 in \cite{Ho} we have that 
$S(11,5)^U=\langle (-396,-100,150,-30)\rangle$, where $\langle (x,w,v,u) \rangle$ denotes the orbit of the solution $(x,w,v,u)$ of \eqref{dickson}, and that 
	$$\Psi_{5,11^5}^*(X)=(X+99)(X+649)(X+979)(X-451)(X-1276).$$
In this way, by Proposition \ref{G5} we have 
$$\mathrm{Spec}(\G(5,11^5))=\{ [n]^1, [255]^{n}, [90]^{n}, [-20]^{n}, [-130]^{n}, [-196]^{n}\}$$ 
where $n=\frac{11^5-1}5=32210$.
\hfill $\lozenge$
\end{exam}

\begin{rem}[\textit{The spectrum of $\G(6,q)$, $\G(8,q)$ and $\G(12,q)$}] \label{g6812q}
The (reduced) period polynomials 
	$$\Psi_{k,q}^*(x)= \prod_{i=1}^k (x-\eta_i^*) \qquad \text{where} \qquad \eta_i^*=k\eta_i+1,$$ 
for $k=6,8,12$ and its factorizations into irreducible polynomials over $\mathbb{Z}$ were obtained by S.\@~Gurak in two papers from 2001 and 2004. 
He first considered the case $q=p^2$ and gave the factorizations (see Propositions 3.1, 3.2 and 3.3 in \cite{Gu1}). The general case is treated in \cite{Gu2}. The results, which are very technical, are given in Propositions 3.2, 3.3 and 3.5 (their descriptions are out of the scope of this paper).
However, since the involved irreducible polynomials are of degree $\le 4$, it is possible in principle to compute all their roots and, hence, to obtain the spectrum of $\G(k,q)$ for $k=6,8,12$.
\end{rem}

\begin{rem}
Using Theorem \ref{Spectro Gkq} one can obtain the spectrum of $\bar \G(k,q)$ and $\G^+(k,q)$ for $k=3,4,5$ enhancing Theorems \ref{gp3q} and \ref{gp4q} and Proposition \ref{G5}. Similarly for the graphs $\G(k,q)$ with $k=6,8,12$ in Remark \ref{g6812q}.	
\end{rem}

We close the section with a comment on Waring numbers $g(k,q)$ over a finite field $\ff_q$. 
The Waring number $g(k,q)$ is defined to be the minimal $g \in \N$ (if it exists) such that every element of $\ff_q$ is a sum of a number $g$ of $k$-th powers in $\ff_q$. It is well-known that $g(k,q) \le k$. 
\begin{rem} 
The Waring number $g(k,q)$ is exactly the diameter of $\G(k,q)$ (see \cite{PV6}). 
For a general graph $G$ one has that diam$(G) \le t-1$, where $t$ is the number of distinct non-principal eigenvalues of $G$ (see for instance Theorem 3.13 in \cite{CDS}), with equality if $G$ is a distance regular graph (see \S \ref{sec:6.2}).
Hence, we have that 
	$$g(k,q) = \mathrm{diam}(\G(k,q)) \le s \le k,$$
where $s$ is the number of distinct Gaussian periods (see \eqref{Gaussian periods diferentes}).
Thus, if $\G(k,q)$ is a distance regular graph then $g(k,q)=s$ (this occurs for instance in the semiprimitive case or in the Hamming case).
\end{rem}

\section{All integral generalized Paley graphs} \label{sec:4}
In this section, we classify those GP-graphs having integral spectrum, by way of period polynomials. 
The study of integral graphs is an interesting topic of research on its own, initiated by Harary and Schwenk back in the 70's 
(see \cite{HS}).

In 1981, Myerson proved that the period polynomial $\Psi_{k,q}(x)$ in \eqref{period polynomial} % = \prod_{i=0}^{k-1} (x-\eta_i^{(k,q)})$ 
has integral coefficients, that is, $\Psi_{k,q}(x)\in \mathbb{Z}[x]$ (see \cite[Theorem 3]{My}). 
Hence, since $\Psi_{k,q}(x)$ is monic we have that all of the Gaussian periods $\eta_{i}^{(k,q)}$ are algebraic integers for all $i=0,\ldots,k-1$. 
Also, he showed that in general the period polynomial 
$\Psi_{k,q}(x)$ splits over $\mathbb{Q}$ in $N$ factors of degree $\frac{k}{N}$ (see \cite[Theorem 4]{My}).
Moreover, he showed that 
\begin{equation} \label{Myerson}
\Psi_{k,q}(x) = \prod_{i=0}^{N-1} \psi^{(i)}_{(k,q)}(x) \qquad \text{with} \qquad \psi^{(i)}_{(k,q)}(x) = \prod_{\ell=0}^{\frac{k}{N}-1}(x-\eta_{i+\ell N}^{(k,q)})\in \mathbb{Z}[x]
\end{equation}
where $\psi^{(i)}_{(k,q)}(x)$ is irreducible or a power of an irreducible polynomial over $\mathbb{Q}$. We will use these facts in the section.

As we have already mentioned in Remark \ref{Spec remarks}, $\G(k,q)$ is integral if and only if $\G^+(k,q)$ is integral (or if $\bar \G(k,q)$ is integral).
By studying the period polynomial of GP-graphs we can now characterize all integral GP-graphs (and hence all integral GP$^+$-graphs). 

\begin{thm} \label{Teo: GP enteros}
Let $q=p^m$ with $p$ prime and $m\in \N$ and let $k\in \N$ such that $k \mid q-1$. 
Then, the generalized Paley graph $\G(k,q)$ is integral if and only if $k$ divides $\tfrac{q-1}{p-1}$; i.e.\@
	\begin{equation} \label{int cond}
		\mathrm{Spec}(\G(k,q)) \subset \mathbb{Z} \qquad \Leftrightarrow \qquad k \mid \tfrac{q-1}{p-1} \qquad \Leftrightarrow \qquad
		\eta_{i}^{(k,q)} \in \mathbb{Z} \quad (0 \le i \le k-1).
	\end{equation}
In particular, all directed GP-graphs are not integral. 
\end{thm}

\begin{proof}
Expression \eqref{spec Gkq} gives the spectra of $\G(k,q)$ in terms of the Gaussian periods $\eta_{i}^{(k,q)}$.
By \eqref{N} and \eqref{int gp} we know that $\eta_i^{(N,q)} \in \Z$ where 
	$N=\gcd(\tfrac{q-1}{p-1},k)$. 
Thus, if $k$ satisfies $k \mid \frac{q-1}{p-1}$ then $k=N$ and hence all the Gaussian periods 
$\eta_i^{(k,q)}$ are integers, by \eqref{int gp}. This implies that $\mathrm{Spec}(\G(k,q))$ is integral. 
	
Now, assume that $\G(k,q)$ is integral, then $\eta_i^{(k,q)}\in \mathbb{Z}$ for all  $i=0,\ldots,k-1$.
Suppose by contradiction that $N<k$. By definition of $N$, there is an integer $L$ such that $k=LN$, i.e.\@ $L=\frac{k}{N}>1$. 
Now, it is known that the Galois group $\textrm{Gal}(\mathbb{Q}(\zeta_p)/\mathbb{Q})$ of the cyclotomic field extension $\mathbb{Q}(\zeta_p)/\mathbb{Q}$ permutes all of the elements in the set 
	$$\{\eta_{i}^{(k,q)}, \eta_{i+N}^{(k,q)},\eta_{i+2N}^{(k,q)},\ldots,\eta_{i+(L-1)N}^{(k,q)} \}$$ 
(see Lemmas 2 and 5 in \cite{My}).
Thus, since $\eta_i^{(k,q)} \in \mathbb{Z}$, we have that $\eta_i^{(k,q)}$ is fixed by all the elements in $\textrm{Gal}(\mathbb{Q}(\zeta_p)/\mathbb{Q})$ and hence we have that $\eta_{i+\ell N}^{(k,q)}= \eta_i^{(k,q)}$ for all $\ell =0,\ldots,L-1$ and so we get
\begin{equation*}
	\psi_{(k,q)}^{(i)}(x) = (x-\eta_{i}^{(k,q)})^L 
\end{equation*}
for $i=0,\ldots, N-1$. Then, by \eqref{Myerson}, we obtain that
\begin{equation}\label{eq: psi power}
	\Psi_{k,q}(x)= p(x)^{L} \qquad \text{where} \qquad p(x)=\prod_{i=0}^{N-1} (x-\eta_{i}^{(k,q)}) \in \mathbb{Z}[x].
\end{equation}
On the other hand, since $-\sum_{i=0}^{k-1}\eta_{i}^{(k,q)}=1$ by \eqref{sum etas=1}, 
we obtain that the term corresponding to $x^{k-1}$ in $\Psi_{k,q}(x)$ is $1$.
Finally, if $b$ denotes the term of $p(x)$ corresponding to $x^{N-1}$, then
the equation \eqref{eq: psi power} implies that $1= bL$ which is absurd since $b,L \in \mathbb{Z}$ and $L>1$.
Therefore, we must have that $k=N$, and hence $k\mid \frac{q-1}{p-1}$, as desired.  
	
The remaining assertion is clear. 	
Indeed, $\G(k,q)$ is directed if and only if $q=p^m$ is odd with $k \nmid \frac{q-1}{2}$.
Thus, if $\G(k,q)$ were integral then $k\mid \frac{q-1}{p-1}$ and since $\frac{q-1}{p-1} \mid \frac{q-1}{2}$, we have that $k \mid
\frac{q-1}{2}$, which is absurd. 
Hence, all directed GP-graphs are non-integral. 
\end{proof}

In the next example we recap integral GP-graphs $\G(k,q)$ with $k=1,2,3,4$ and check the arithmetic condition \eqref{int cond} in the theorem.

\begin{exam} \label{spec G2q}
($i$) 
In Example \ref{g1q} we saw that the graphs $\G(1,q)=K_q$ are integral and the condition $1\mid \frac{q-1}{p-1}$ is trivial. 

\noindent ($ii$) 
In Example \ref{Ex:Paley} we saw that the graphs $\G(2,q)$ are integral if and only if $q\equiv 1 \pmod 4$ (hence $p\equiv 1 \pmod 4$ or $p\equiv 3 \pmod 4$ and $m=2t$), that is when $\G(2,q)$ is the classic Paley graph. It is easy to see that $2\mid \tfrac{q-1}{p-1}$ if and only if $p\equiv 1 \pmod 4$ or $p\equiv 3 \pmod 4$ and $m=2t$. 

\noindent ($iii$) 
In Theorem \ref{gp3q} we showed that $\G(3,q)$ is integral for $p\equiv 1 \pmod 3$ with $3\mid m$ and for $p\equiv 2 \pmod 3$ and $m$ even (and not for $p\equiv 1 \pmod 3$ with $3 \nmid m$). 
It is easy to see that these conditions are equivalent to 
$3\mid \tfrac{q-1}{p-1}$. 

\noindent ($iv$) 
In Theorem \ref{gp4q} we showed that $\G(4,q)$ is integral for $p\equiv 1 \pmod 4$ with $m\equiv 0 \pmod 4$ or $p\equiv 3\pmod 4$ and not for $p\equiv 1 \pmod 4$ with $m\equiv 2 \pmod 4$. One can check that $4\mid \tfrac{q-1}{p-1}$ if and only if $p\equiv 1 \pmod 4$ with $m\equiv 0 \pmod 4$ or $p\equiv 3\pmod 4$.
\hfill $\lozenge$
\end{exam}

\begin{exam} \label{Ex Hamming}
A Hamming graph $H(b,q)$ is a graph with vertex set $V=K^b$ where $K$ is any set of size $q$
(typically $\ff_q$ in applications), and where two $b$-tuples form an edge if and only if they differ in exactly one coordinate.
Notice that $H(b,q)=\square^b K_{q}$ and hence, Hamming graphs are integral with spectrum given by 
\begin{equation*} \label{Spec Hamming}
	\mathrm{Spec}(H(b,q)) = \big\{ [\ell q-b]^{\binom{b}{\ell}(q-1)^{b-\ell}} : 0\le \ell \le b \big\}. 
\end{equation*}	
Those connected GP-graphs $\G(k,q)$ which are Hamming were classified by Lim and Praeger in \cite{LP}. In this case $k=\tfrac{p^{bm}-1}{b(p^m-1)}$ for integers $b \mid \tfrac{p^{bm}-1}{p^m-1}$, $q=p^{bm}$, and 
\begin{equation*} \label{Hamming}
	\G(\tfrac{p^{bm}-1}{b(p^m-1)}, p^{bm}) = H(b,p^{m}).
\end{equation*}	
It is clear that $\tfrac{p^{bm}-1}{b(p^m-1)} \mid \frac{p^{bm}-1}{p-1}$ and hence Theorem \ref{Teo: GP enteros} implies that $\G(\tfrac{p^{bm}-1}{b(p^m-1)}, p^{bm})$ is an integral graph. \hfill $\lozenge$
\end{exam}

When $\mathrm{Spec}(\G(k,q))$ is not integral, the graph has at least one irrational Gaussian period.
\begin{coro}
	Let $q=p^m$ with $p$ prime and let $k\in \N$ such that $k \mid q-1$.
	If $k \nmid \frac{q-1}{p-1}$, then there exists at least one $j \in \{0,\ldots,k-1\}$ such that $\eta_{j}^{(k,q)} \not \in \mathbb{Q}$.
\end{coro}

\begin{proof}
	Since $k\nmid \frac{q-1}{p-1}$ we know that there is some $j\in \{0,\ldots,k-1\}$ such that 
	$\eta_{j}^{(k,q)} \not \in \mathbb{Z}$, by \eqref{int cond}. %Corollary \ref{IntGP}. 
	Since the Gaussian periods are algebraic integers, if $\eta_{j}^{(k,q)} \in \mathbb{Q}$ then $\eta_{j}^{(k,q)} \in \mathbb{Z}$. Thus, there exists some $j\in \{0,\ldots,k-1\}$ such that $\eta_{j}^{(k,q)} \not \in \mathbb{Q}$, as we wanted to see. 
\end{proof}

Let $q=p^m$ with $p$ prime, let $k\in \N$, and consider the following condition 
\begin{equation} \label{conditions}
	p \equiv 1 \!\!\pmod k \quad \text{and} \quad k\mid m \qquad \quad \text{or} \qquad \quad p\not \equiv 1 \pmod k.
\end{equation} 
The following corollary of Theorem \ref{Teo: GP enteros} characterizes integral GP-graphs $\G(k,q)$ in terms of condition \eqref{conditions}.

\begin{coro}\label{Coro flia int p1}
Let $\G(k,q)$ be a GP-graph with $q=p^m$ and $p$ prime. 
If $p\equiv 1 \pmod k$ then $\mathrm{Spec}(\G(k,q)) \subset \Z$ if and only if $k\mid m$.
Furthermore, we have:
\begin{enumerate}[$(a)$]
	\item If $\mathrm{Spec}(\G(k,q)) \subset \Z$ then condition \eqref{conditions} holds. \msk
	
	\item If $k$ is prime and condition \eqref{conditions} holds then $\mathrm{Spec}(\G(k,q)) \subset \Z$.
\end{enumerate}
In particular, if $k$ is prime then $\mathrm{Spec}(\G(k,q)) \subset \Z$ if and only if condition \eqref{conditions} holds.
\end{coro} 

\begin{proof}
By Theorem \ref{Teo: GP enteros}, $\mathrm{Spec}(\G(k,q)) \subset \Z$ if and only if $k\mid\frac{q-1}{p-1}$.
Notice that $k\mid \frac{q-1}{p-1}$ if and only $k\mid \Psi_{m}(p)$, where 
	$\Psi_{m}(p) = p^{m-1} + \cdots + p^2+p+1$
Thus, if $p\equiv 1 \pmod{k}$ we have that $\Psi_{m}(p)\equiv m \pmod{k}$. That is, $k\mid \Psi_{m}(p)$ if and only if $k\mid m$.

\noindent ($a$) 
Suppose that $\G(k,q)$ is integral. There are two possibilities: $k\mid p-1$ or not.
In the first case, we known that $k\mid m$, as we wanted to show.

\noindent ($b$) 
By the first part of the statement, it is enough to check the claim for $p\not \equiv 1\pmod{k}$ since otherwise we know that $\G(k,q)$ is integral. 
Thus, assume that $p\not \equiv 1\pmod{k}$. Since $k$ is prime by hypothesis with $k\nmid p-1$ and $k\mid q-1$, then $k\mid\frac{q-1}{p-1}$. Therefore, $\G(k,q)$ is integral as desired.
The remaining statement is straightforward.
\end{proof}

\begin{exam}
The graphs $\G(5,p^{5t})$ with $t \in \N$ and $p$ prime of the form $5\ell+1$ (studied in Proposition \ref{G5}) have integral spectrum (notice that this is not clear at all from the expressions of the Gaussian periods in Proposition \ref{G5}). 
\hfill $\lozenge$	
\end{exam}

We close the section with a result on divisibility of the energy.
\begin{coro}
	If $\G(k,q)$ is an integral GP-graph then the degree of regularity $n=\frac{q-1}k$ divides the energy of $\G(k,q)$ and of $\G^+(k,q)$, that is 
	$n\mid E(\G^*(k,q))$. 
\end{coro}

\begin{proof}
By \eqref{spec Gkq} and \eqref{spec Gkq+}, the energy of $\G^*(k,q)$ is given by
	$$E(\G^*(k,q))=n(1+\mu n + \sum_{i=1}^s \mu_i |\eta_i|)$$
where $\mu, \mu_i \in  \N_0$ and $\eta_i \in \Z$ %are integers 
for $i=1,\ldots,s$ by hypothesis, and this implies the result.	
\end{proof}

We point out that for regular graphs which are not GP-graphs, the result does not hold in general. For instance, the cubic graph $C_6^*$ which is the $6$-cycle with loops, has spectrum $\{[3]^1, [2]^2, [0]^2, [-1]^1\}$ and hence energy $E(C_6^*)=8$ and $3\nmid 8$. For an example without loops, consider the cubic graph of six vertices numbered $\G_{51}$ in \cite{CvP}. This graph has spectrum $\{[3]^1, [1]^1, [0]^2, [-2]^2\}$ and hence energy $E(\G_{51})=8$ and $3\nmid 8$ (see Table 1 in \cite{PV8}).

\section{Semiprimitive generalized Paley graphs} \label{sec:5}
In this and the next section we focus on a particular family of GP-graphs, the semiprimitive ones. 
Let $\G(k,q)$ with $q=p^m$ and $k\mid q-1$.

In the study of 2-weight irreducible cyclic codes, the \textit{semiprimitive case} corresponds to 
$-1$ being a power of a prime $p$ modulo $k$ (see \cite{SW}). 
If $t \in \N$ is minimal such that 
	$$p^{t}\equiv -1 \pmod{k},$$ 
then $\mathrm{ord}_{k}(p) = 2t$ when $k>2$ and so, since $q=p^m\equiv 1 \pmod{k}$, 
we obtain that 
	$$m=2ts$$ 
for some positive integer $s$ when $k>2$.
Then, we have that semiprimitiveness is equivalent to either $k=2$ and $q$ odd or else $k>2$ and  
\begin{equation} \label{semip cond}
	k \mid p^t+1 \quad \text{ for some } \quad t \mid \tfrac m2.
\end{equation}

With respect to the GP-graphs in the semiprimitive case, notice that if $k=2$ with $q$ odd, $\G(2,q)$ is non-directed if $q\equiv 3 \pmod{4}$ and directed if $q\equiv 1 \pmod{4}$. On the other hand, if $k>2$ then the graph 
$\G(k,q)$ is always undirected, since by assumptions, if $m=2s$ then $k\mid \tfrac{q-1}2$ since 
	$$\tfrac12(q-1) = \tfrac 12 (p^s-1)(p^s+1).$$
Furthermore, if $k=p^{\frac m2}+1$ then $\G(k,q)$ is not connected (see Proposition~4.6 in \cite{PV1}). Indeed, 
one can prove that
$\G(p^{\frac{m}{2}+1},p^{m}) \cong K_{p^{\frac{m}{2}}} \cup \cdots \cup K_{p^{\frac m2}}$ 
($p^{\frac m2}$-times).

\begin{defi} \label{def semip}
We say that $(k, q)$ with $k=2$ and $q\equiv 1 \pmod 4$ or $k>2$ satisfying \eqref{semip cond} and $k \ne p^{\frac m2}+1$ is a \textit{semiprimitive pair} of integers. 
If $(k,q)$ is a semiprimitive pair of integers, we will refer to $\G(k,q)$ as a \textit{semiprimitive GP-graph}.
Hence, every semiprimitive GP-graph $\G(k,q)$ is undirected and connected.
\end{defi}

For instance, if $p=3$ and $m=4$, to find the semiprimitive pairs of the form $(k,81)$ we take $k\mid 3^2+1=2\cdot 5$ and $k\mid 3^1+1=4$. Hence $k=2, 4$ or $5$, while $k=10$ is not allowed since $10=3^{\frac m2}+1$.

\begin{rem} \label{semipr}
\noindent ($i$) 
Three infinite families of semiprimitive pairs, for $p$ prime and $m=2t \ge 2$, with $k=2,3,4$ respectively are given by:   
\begin{enumerate}[($a$)] 
	\item the pairs $(2,p^{2t})$ with $p$ odd; \sk 
			
	\item the pairs $(3,p^{2t})$ with $p\equiv 2 \pmod 3$ and $t\ge 1$ (where $t\ge 2$ if $p=2$); \sk 
		
	\item the pairs $(4,p^{2t})$ with $p\equiv 3 \pmod 4$ and $t\ge 1$ (where $t\ge 2$ if $p=3$). 
\end{enumerate}
The first of the three families of pairs give rise to the classical Paley graphs $\G(2,p^{2t})$.
		
\noindent ($ii$) 
Another infinite family of semiprimitive pairs is given by $(p^\ell+1,p^m)$ 
with $p$ prime, $m\ge 2$, $\ell\mid m$ and $\frac{m}{\ell}$ even. 
They give the GP-graphs $\G(q^\ell+1,q^m)$, with $q=p$, considered in \cite{PV1} for $q$ a power of $p$.
Notice that the graphs $\G(3, 2^{2t})$ with $t\ge 2$ and $\G(4,3^{2t})$ with $t\ge 1$ belong to both families given in ($i$) and ($ii$). For instance, $\G(3,16)$, $\G(3,64)$, and $\G(4,81)$ are semiprimitive GP-graphs.	
\end{rem}

Using the previous definition and items ($i$) and ($ii$) in the remark, we give a list of the smallest semiprimitive pairs $(k,q)$ with $q=p^m$ for $p=2,3,5,7$ and $m=2,4,6,8$.  
\begin{table}[h!]
	\caption{Values of $k$ for small semiprimitive pairs $(k,p^m)$.}
	 \label{tabla 1}
	 \begin{tabular}{|c|p{1.25cm}|p{3cm}|p{3.25cm}|p{3.25cm}|} \hline 
		& $m=2$ & $m=4$ & $m=6$ & $m=8$ \\ \hline
		$p=2$ & --    & 3 & 3 & 5 \\ \hline 
		$p=3$ & --    & 2, 4, \textbf{5} & 2, 4, \textbf{7}, \textbf{14} & 2, 4, \textbf{5}, \textbf{10}, \textbf{41} \\ \hline 
		$p=5$ & 2, \textbf{3}  & 2, \textbf{3}, 6, \textbf{13} & 2, \textbf{3}, 6, \textbf{7}, \textbf{9}, \textbf{14}, \textbf{18}, 
		\textbf{21},  \textbf{42}, \textbf{63} & 2, \textbf{3}, 6, \textbf{13}, 26, \textbf{313} \\ \hline
		$p=7$ & 2, \textbf{4}  & 2, \textbf{4}, \textbf{5}, 8, \textbf{10}, \textbf{25} & 2, \textbf{4}, \textbf{5}, 8, 
		\textbf{10}, \textbf{25}, \textbf{43}, 50, \textbf{86}, \textbf{172} & 2, \textbf{4}, \textbf{5}, 8, \textbf{10}, \textbf{25}, 50, 
		\textbf{1201} \\ \hline
	\end{tabular}
\end{table}

Here we have marked in bold those $k$ which are different from 2 and not of the type $p^\ell+1$ for some $p$ and $\ell$, showing that in general there are much more semiprimitive graphs $\G(k,q)$ than Paley graphs $\G(2,q)$ or GP-graphs of the form $\G(p^\ell+1,p^m)$. 

It is well-known that the Gaussian periods associated to a semiprimitive pair $(k,q)$ are integers (see for instance \cite{DY}) and hence $\G(k,q)$ is integral. 
We now use Theorem \ref{Teo: GP enteros} to obtain the same result in an indirect but elementary way (i.e.\@ without explicitly computing the spectrum of $\G(k,q)$). 
\begin{prop} \label{Coro Semip GP integrals}
	Every semiprimitive GP-graph $\G(k,q)$ is integral.
\end{prop}
	
\begin{proof}
By Theorem \ref{Teo: GP enteros}, $\G(k,q)$ is integral if and only if $k\mid\frac{q-1}{p-1}$, where $q=p^m$ for some $m$.
Thus, we will show that if $(k,q)$ is a semiprimitive pair then $k\mid\frac{q-1}{p-1}$. 
If $t$ is the minimal positive integer such that $p^{t}\equiv -1 \pmod{k}$, then $2t = \mathrm{ord}_{k}(p)$ and so, since 
$q=p^m\equiv 1 \pmod{k}$, we obtain that $m=2ts$ for some positive integer $s$.

Notice that we have the factorization 
	$$q-1= p^{2ts}-1=(p^t-1)\Psi_{2s}(p^t)$$ 
where $\Psi_{2s}(x) = x^{2s-1} + \cdots +x^2+x+1$.
Since $2s$ is even and $p^t \equiv -1\pmod{k}$ we obtain that $k\mid \Psi_{2s}(p^t)$. 
On the other hand, $p-1\mid p^t-1$ trivially, and hence we have that $k\mid \frac{q-1}{p-1}$. 
		
We have shown that $k\mid \frac{q-1}{p-1}$ for every semiprimitive pair $(k,q)$ and 
therefore $\G(k,q)$ is integral, as we wanted to see. 
\end{proof}

\subsection{The spectrum of semiprimitive GP-graphs $\G(k,q)$} \label{sec:6.1}
In this subsection we recall the spectrum for arbitrary semiprimitive GP-graphs. 
In 1999, by using Gauss sums, Brouwer, Wilson and Xiang computed the spectra of a more general family defined in terms of semiprimitive pairs (see Theorem 2 in \cite{BWX}). Now, for completeness, using Gaussian periods 
we give the spectrum of the corresponding GP$^*$-graphs $\G(k,q)$ and $\G^+(k,q)$ and of the complements $\bar \G(k,q)$.

We will need the following notation.
If $q=p^m$, define the sign 
	\begin{equation} \label{signo} 
		\sigma = (-1)^{s+1}
	\end{equation}
where $s=\frac{m}{2t}$ and $t$ is the least integer $j$ such that $k \mid p^j+1$ (hence $s \ge 1$).

\begin{thm} \label{semiprimitive}
Let $(k, q)$ be a semiprimitive pair with $q=p^m$, $m$ even, and put $n=\frac{q-1}k$.  
Then, the spectra of $\G=\G(k,q)$, $\G^+=\G^+(k,q)$ and $\bar \G = \bar \G(k,q)$ are integral and respectively given by 	\begin{equation*}
		\begin{aligned}
			& \mathrm{Spec}(\G) 		 = \{ [n]^1, [\lambda_1]^n, [\lambda_2]^{(k-1)n}\}, \\[1.5mm] 
			& \mathrm{Spec}(\bar \G)  = \{[(k-1)n]^1, [(k-1) \lambda_2]^n, [-1-\lambda_2]^{(k-1)n}\},
		\end{aligned}
	\end{equation*}
where 
	\begin{equation} \label{autoval semip}
		\lambda_1 = \frac{\sigma (k-1) p^{\frac m2}-1}{k}
		\qquad  \text{and} \qquad  
		\lambda_2 = -\frac{\sigma p^{\frac m2}+1}{k} 
	\end{equation}
with $\sigma$ as given in \eqref{signo}. 
Furthermore, we have $\mathrm{Spec}(\G^+) = \mathrm{Spec}(\G)$ if $q$ is even and 
$\mathrm{Spec}(\G^+) = \{ [n]^1, [\pm \lambda_1]^{\frac n2}, [\pm \lambda_2]^{\frac{(k-1)n}2}\}$ if $q$ is odd. 
\end{thm}

\begin{proof}
We first compute the spectrum of $\G = \G(k,q)$, which by Theorem \ref{Spectro Gkq} is given in terms of Gaussian periods. 
From Lemma 13 in \cite{DY} the Gaussian periods $\eta_j^{(k,q)}$, for $j=0, \ldots, k-1$, are given by:
		
\noindent ($a$) 
If $p$, $\alpha=\frac{p^t+1}k$ and $s$ are all odd then 
	\begin{equation*} \label{gp1} 
		\eta_j^{(k,q)} = \begin{cases}
		\frac{(k-1)\sqrt q -1}k & \qquad \text{if } j=\frac k2, \\[1mm]
		-\frac{\sqrt q +1}k & \qquad \text{if } j \ne \frac k2.
		\end{cases}
	\end{equation*}
		
\noindent ($b$) 
In any other case we have $\sigma=(-1)^{s+1}$ and 
	\begin{equation*} \label{gp2} 
		\eta_j^{(k,q)} = \begin{cases}
		\frac{\sigma (k-1)\sqrt q -1}k & \qquad \text{if } j=0, \\[1mm]
		-\frac{\sigma \sqrt q +1}k & \qquad \text{if } j \ne 0.
	\end{cases}
	\end{equation*}

Thus, by Theorem \ref{Spectro Gkq}, the spectrum of $\G(k,q)$ is 
$\mathrm{Spec}(\G(k,q)) = \{ [n]^1, [\eta_{k/2}]^n, [\eta_0]^{(k-1)n}\}$ if 
$p$, $\alpha$, $s$ are odd or $\mathrm{Spec}(\G(k,q)) = \{ [n]^1, [\eta_0]^n, [\eta_1]^{(k-1)n}\}$ otherwise.
		
Suppose we are in case ($a$), i.e.\@ $p$, $\alpha$ and $s$ are odd. 
Then we have 
	$$\lambda_1 = \eta_{k/2} =  \tfrac{(k-1)p^{\frac m2}+1}k 
		\qquad  \text{and} \qquad 
	\lambda_2 = \eta_j = \eta_0 = -\tfrac{p^{\frac m2}+1}{k} \quad (j\ne \tfrac k2).$$
It is clear that $\lambda_2\ne n$ and $\lambda_2\ne \lambda_1$. Also, $n \ne \lambda_1$ since 
$k \ne p^{\frac m2}+1$. 
Thus, all three eigenvalues are different and their corresponding multiplicities are as given in the statement. 
		
In case ($b$), we have 
	$$\eta_0 = \tfrac{\sigma(k-1)p^{\frac m2}+1}k \qquad  \text{and} \qquad  
	\eta_j = -\tfrac{\sigma p^{\frac m2}-1}k \quad (j\ne 0).$$ 
Again, one checks that $\eta_0 \ne \eta_j$, $\eta_0\ne n$ and $\eta_j \ne n$ for every $j\ne \frac k2$.  
Thus, the corresponding multiplicities are as stated in the proposition. 
		
Combining cases ($a$) and ($b$) we get \eqref{autoval semip}. %, as we wanted to show.
Finally, the spectra of $\bar \G(k,q)$ and $\G^+(k,q)$ follow by Theorem \ref{Spectro Gkq}.
Just recall that for $q$ even  we have that $\G^+=\G$.
\end{proof}
	
Notice that for $(k,q)$ semiprimitive with $q$ odd, $\G^+(k,q)$ has almost symmetric spectrum (see Definition 2.13 in \cite{PV5}) with five different eigenvalues.

\noindent
\textit{Note.} Since $\lambda_1, \lambda_2 \in \Z$, we have that 
$\sigma=\pm 1$ if and only if $k \mid p^{\frac m2} \pm 1$, respectively.

\begin{rem}
The weight distribution of $2$-weight irreducible cyclic codes $\CC(k,q)$ in the semiprimitive case is known (see for instance \cite{SW}). Using this and the relation of the weight distribution of $\CC(k,q)$ with the spectrum of GP-graphs obtained in 
Theorem~5.4 in \cite{PV7} one can also recover the spectrum of semiprimitive GP-graphs $\G(k,q)$ 
as in \eqref{spec Gkq} in Theorem \ref{Spectro Gkq}.	 
\end{rem}
	
\begin{rem}
We have computed the spectrum of the GP-graphs $\G_{q,m}(\ell)=\G(q^\ell+1,q^m)$ and $\bar \G_{q,m}(\ell)$, 
with $\ell \mid m$ and $\frac{m}{\ell}$ even (see Theorem~3.5 and Proposition~4.3 in \cite{PV1}, see also \cite{PV2}), 
by using certain sums associated with the quadratic forms 
	$$Q_{\gamma,\ell}(x) = \Tr_{p^m/p} (\gamma x^{q^\ell+1})$$ 
with $\gamma \in \ff_{p^m}^*$.
By ($ii$) in Example \ref{semipr}, the graph $\G(p^\ell+1,p^m)$, i.e.\@ with $q=p$ prime, 
is semiprimitive and hence its spectrum is given by Theorem \ref{Spectro Gkq}. Indeed, 
$\mathrm{Spec}(\G(p^\ell+1,p^m))=\{ [n]^1, [\lambda_1]^{n}, [\lambda_2]^{p^\ell n} \}$ 
where 
\begin{equation*} \label{spec pl+1} 
	n = \tfrac{p^m-1}{p^\ell+1}, \qquad \quad \lambda_1 = \tfrac{\sigma p^{\frac m2 + \ell}-1}{p^\ell+1}, \qquad \quad 
	\lambda_2 = -\tfrac{\sigma p^{\frac m2}+1}{p^\ell+1},
\end{equation*}
with $\sigma = (-1)^{\frac{m}{2\ell}+1}$. 
It is reassuring that both computations of the spectrum coincide after using these two different methods. The same happens for the complementary graphs. 
\end{rem}

We conclude the section showing that for each odd prime power $p^{2m}$ there is only one semiprimitive GP-graph $\G(k_m,p^{2m})$ which is Hamming (or equivalently, there is only one Hamming GP-graph which is semiprimitive). 
	
\begin{prop}
Let $\G(k,q)$ be a semiprimitive GP-graph with $q=p^t$ and $p$ an odd prime. Then, $\G(k,q)$ is Hamming if and only if $k=\frac{p^m+1}{2}$ and $t=2m$. In this case we have
	$$ \G(\tfrac{p^m+1}2,p^{2m}) = H(2,p^m)= K_{p^m} \square K_{p^m} = L_{q,q}$$ 
where $L_{q,q}$ is the $q\times q$ lattice (or rook's) graph, with integral spectrum given by
	$$\textrm{Spec}(\G(\tfrac{p^m+1}2,p^{2m})) = \{[2(p^m-1)]^1, [p^m-2]^{2(p^m-1)}, [-2]^{(p^m-1)^2}\}.$$
\end{prop}

\begin{proof}
In Example \ref{Ex Hamming} we recall that Hamming GP-graphs (classified in \cite{LP}) are of the form
	$$\G(\tfrac{p^{bm}-1}{b(p^m-1)}, p^{bm}) = H(b,p^{m})= \square^b K_{p^m}$$ 
with spectrum 
\begin{equation} \label{eq: spec hamming}
	\{ [ \ell p^{m}-b ]^{\binom{b}{\ell}(p^{m}-1)^{b-\ell}} \}_{\ell=0}^b.
\end{equation}
Since semiprimitive graphs have exactly three eigenvalues, we must necessarily have that $b=2$ and we check that in this case
	$$\G(\tfrac{p^{2m}-1}{2(p^m-1)}, p^{2m}) = \G(\tfrac{p^{m}+1}2, p^{2m})$$
is semiprimitive since $\tfrac{p^{m}+1}2 \mid p^m+1$.
It is well-known that $H(2,q)$ is the lattice graph $L_{q,q}$. The spectrum follows by taking $b=2$ in  
(or one can also use Theorem~\ref{semiprimitive}).
\end{proof}

\subsection{Semiprimitive GP-graphs are strongly regular} \label{sec:6.2}
Let $\G$ be a regular graph that is neither complete nor empty. Then $\G$ is
said to be \textit{strongly regular} with parameters $srg(v, r, e, d)$
if it is $r$-regular with $v$ vertices, every pair of adjacent vertices has $e$ common neighbours,
and every pair of distinct non-adjacent vertices has $d$ common neighbours. These parameters are tied by the relation
\begin{equation}\label{main relation}
	(v-r-1)d=r(r-e-1).
\end{equation}
For instance, for $q\equiv 1 \pmod 4$, the classic Paley graph $P(q)$ is a strongly regular graph with parameters
$srg(q, \tfrac{q-1}2, \tfrac{q-5}4, \tfrac{q-1}4)$.
If $\G$ is strongly regular with parameters
$srg(v, r, e, d)$, 
then its complement $\bar \G$ is also strongly regular with parameters
$srg(v, \bar r, \bar e, \bar d)$, where 
\begin{equation}\label{srg comp}
	\bar r = v-r-1, \qquad \bar e= v-2-2r+d \qquad \text{and} \qquad \bar d= v-2r+e.
\end{equation}	

Let $\G=srg(v,r,e,d)$. If $2r-(v-1)(e-d)\ne 0$ the graph have integral different eigenvalues. On the other hand, if 
	$$2r-(v-1)(e-d)= 0$$ 
the graph is said to be a \textit{conference graph} because of their connection with symmetric conference matrices. A conference graph has parameters 
	$$srg(v,\tfrac{v-1}2, \tfrac{v-5}4, \tfrac{v-1}4).$$ 
Hence, Paley graphs are conference graphs. 

A strongly regular graph $srg(v,n,e,d)$ with different eigenvalues $n,f,g$ is a \textit{pseudo-Latin square graph} if $n=-g(f-g-1)$, where $n>f>0>g$. Equivalently, it is denoted $PL_\delta(w)$ and has parameters 
\begin{equation} \label{PL}
	PL_\delta(w) = srg(w^2, \delta(w-1), \delta^2-3\delta+w, \delta(\delta-1)),
\end{equation}	
where $w=f-g$ and $\delta=-g$. A graph with the parameters as above, changing $\delta$ and $w$ by $-\delta$ and $-w$ is called a \textit{negative Latin square graph}. It is denoted by $NL_\delta(w)$ and has the parameters
\begin{equation} \label{NL}
	NL_\delta(w) = srg(w^2, \delta(w+1), \delta^2+3\delta-w, \delta(\delta+1)).
\end{equation}
See for instance Chapter 8 in \cite{CvL} for the definitions of pseudo-Latin and negative Latin square graphs.

A regular graph is called \textit{distance regular} if for any two vertices $v$ and $w$, the number of vertices at distance $j$ from $v$ and at distance $k$ from $w$ depends only upon $j, k$, and the distance $d(v,w)$ between $v$ and $w$.	
A connected strongly regular graph $\Gamma$, being a distance regular graph of diameter $\delta =2$, have intersection array of the form 
	$$\mathcal{A}(\Gamma) = \{b_0,b_1, b_2; c_0, c_1, c_2\}.$$ 
For every $i=0,1,2$ and every pair of vertices $x,y$ at distance $i$, the intersection numbers are defined by 
	$$b_i = \#\{z\in N(y): d(x,z)=i+1 \} \qquad \text{and} \qquad c_i = \#\{ z\in N(y) : d(x,z)=i-1 \},$$ 
where $N(y)$ denotes the set of neighbours of $y$. Since we trivially have $b_2=c_0=0$, we will simply write $\mathcal{A}(\Gamma) = \{b_0,b_1;  c_1, c_2\}$, as usual.	
See \cite{BrV} for an introduction and general examples about strongly regular graphs.

We now give some structural properties of the graphs $\G(k,q)$ throughout the spectrum.
\begin{thm} \label{srg}
Let $(k, q)$ be a semiprimitive pair with $q=p^m$, $m=2ts$ where $t$ is the least integer satisfying $k\mid p^t+1$ and put 
$n=\frac{q-1}k$. Then we have:
  \begin{enumerate}[$(a)$]
	\item $\G(k,q)$ and $\bar \G(k,q)$ are primitive, non-bipartite, integral, strongly regular graphs with corresponding parameters 
			$srg(q,n,e,d)$ and $srg(q,(k-1)n,e',d')$ given by 
			$$e=d +  (\sigma p^{\frac m2} + 2\lambda_2), \quad d=n + (p^{\frac m2} +\lambda_2) \lambda_2),  
			\quad e'=q-2-2n+d, \quad d'=q-2n + e.$$
			
	\smallskip 
	\item $\G(k,q)$ and $\bar \G(k,q)$ are distance regular graphs of diameter 2 with intersection arrays 
			$$\mathcal{A} = \{n,n-e-1;1,d\} \qquad \text{and} \qquad 
			\bar{\mathcal{A}} = \{(k-1)q, n-d; 1, q-2n+e\}.$$
			
	\smallskip
	\item If $s$ is odd then $\G(k,q)$ and $\bar \G(k,q)$ are pseudo-Latin square graphs with parameters
		\begin{equation*} \label{LSquares} 
			PL_\delta(w) = srg(w^2, \delta(w-1), \delta^2-3\delta+w, \delta(\delta-1)),
		\end{equation*}
		where $w=f-g$, $\delta=-g$ and $f>0>g$ are the non-trivial eigenvalues of $\G(k,q)$ or $\bar \G(k,q)$. 
  \end{enumerate}
\end{thm}
	
\begin{proof}
We will use the spectral information from Theorem \ref{semiprimitive}. 
We prove first the results for $\G(k,q)$.
		
\noindent ($a$) 
Since the multiplicity of the degree of regularity $n$ is $1$, the graph is connected. 
Also, one can check that $-n$ is not an eigenvalue of $\G(k,q)$ and hence the graph is non-bipartite.
Now, since $k\mid p^t+1$ then $k\mid p^{ts}+1$ if $s$ is odd and $k\mid p^{ts}-1$ if $s$ is even, 
hence 
$\beta = \frac{p^{\frac m2}+\sigma}k$ is an integer. Thus, since 
$\lambda_1=\sigma(p^{ts}-\beta)$ and $\lambda_2=-\sigma \beta$, by \eqref{autoval semip}, the eigenvalues are all integers (we also know this from Theorem \ref{Spectro Gkq}).
		
Finally, since the graph is connected, $n$-regular with $q$-vertices and has exactly three eigenvalues, it is a strongly regular graph with parameters $srg(q,n,e,d)$. 
We now compute $e$ and $d$. It is known that the non-trivial eigenvalues of an strongly regular graph are of the form 
	$$\lambda^\pm = \tfrac 12 \{ (e-d) \pm \Delta \} \qquad \text{where} \qquad \Delta = \sqrt{(e-d)^2 + 4(n-d)}.$$  
Thus, $d= n +\lambda^ + \lambda^-$ and $e = d + \lambda^+ + \lambda^-$.
From this and \eqref{autoval semip} the result follows.  
		
\noindent ($b$) 
We know that $\Gamma = srg(q,n,e,d)$ is primitive for $(k,q)$ a semiprimitive pair. 
Since $\Gamma$ is connected with diameter $\delta=2$, its intersection array is $\{ n, n-e-1; 1, d \}$. 
In fact, it is clear that $b_0=n$ and $c_1=1$. Let $x,y$ be vertices of $\Gamma$. 
Thus, if $d(x,y)=1$, then 
$$b_1 = \#(N(y) \smallsetminus\{x\}) - \#N(x) =n-1-e.$$ 
If $d(x,y)=2$, then 
$c_2 = \#(N(x) \cap N(y))$. 
Since $\bar \Gamma$ is also connected with diameter $2$, its intersection array is $\{ \bar n, \bar n - \bar e-1; 1, \bar d \}$.
Now, since 
	$$q-2n+e = (q-n+1)-(n-e+1),$$ 
by using \eqref{main relation} and \eqref{srg comp} we get the desired result.
		
\noindent ($c$) 
Note that the regularity degree of $\G=\G(k,q)$ equals the multiplicity of a non-trivial eigenvalue by Theorem \ref{semiprimitive}. Thus, Proposition 8.14 in \cite{CvL}, we have that $\G$ is of pseudo-Latin square type graph (PL), of negative Latin square type (NL) or is a conference graph.
By definition, a conference graph satisfy 
	$$2n+(q-1)(e-d)=0.$$
It is easy to check that this condition holds for $\G(k,q)$ if and only if 
$\G(1,4)=K_4$, and hence $\G$ is not a conference graph. 
Now, put $w=f-g$, where $f,g$ are the non-trivial eigenvalues with $f>0>g$ and 
$\delta=-g$. Then, $\G$ is a pseudo-Latin square graph with parameters as in \eqref{PL}  
or a negative Latin square graph with parameters as in \eqref{NL}.
It is clear that 
	$$n=\delta(w-1)$$ if and only if $s$ is odd and that for $s$ even $n \ne \delta(w+1)$.
For $\bar \G$ one proceeds similarly. Hence the only possibility for $\G$ and $\bar \G$ is to be pseudo-Latin square graphs.

Now, it is easy to see that $\bar \G(k,q)$ is also a primitive non-bipartite integral strongly regular graph 
with parameters and intersection array as stated. 
The proof that $\bar \G(k,q)$ is a pseudo-Latin square if $s$ is odd is analogous to the previous one for $\G(k,q)$ and we omit the details. 
Finally, since $\G(k,q)$ is a pseudo-Latin  square graph $PL_\delta(w)$ then $\bar \G(k,q)$ is a pseudo-Latin square graph $PL_{\delta'}(w)$ with 
$\delta'=u+1-\delta$ (see \cite{CvL}). 
\end{proof}
	
\begin{rem}
Notice that if we take 
	$h = \min\{|f|, |g|\}$, 
then for $s$ even (in the previous notations), $\G(k,q)$ satisfy the same parameters as in \eqref{NL} with $\delta$ replaced by $h$, that is $\G(k,q)$ is a strongly regular graph with parameters, 
in terms of the eigenvalues, given by
	$$\widetilde{NL} = srg(w^2, h(w+1), h^2+3h-w, h(h+1)).$$
\end{rem}

\begin{exam} \label{tablitas srg}
From Theorem \ref{semiprimitive} and Proposition \ref{srg} we obtain Table \ref{tablita srg} below. 
Here $s=\frac{m}{2t}$ where $t$ is the least integer such that $k\mid p^t+1$.
\renewcommand{\arraystretch}{1.1}
\begin{small}
	\begin{table}[h!] 
	\caption{Smallest semiprimitive graphs: srg parameters and spectra}  \label{tablita srg}
	\begin{tabular}{|c|c|c|c|c|c|} \hline 
	graph 			 & srg parameters & spectrum & $t$ & $s$ & pseudo-latin square \\ \hline
					$\G(3,2^4)$  & $(16, 5, 0, 2)$   &  $\{ [5]^1, [1]^{10}, [-3]^{5} \}$ & 1& 2& no \\ 
					$\bar\G(3,2^4)$  & $(16, 10, 6, 6)$   &  $\{ [10]^1, [2]^{5}, [-2]^{10} \}$ & 1& 2& no \\ \hline
					$\G(3,2^6)$  & $(64, 21, 8, 6)$  &  $\{ [21]^1, [5]^{21}, [-3]^{42} \}$ & 1& 3& $PL_3(8)$ \\ 
					$\bar\G(3,2^6)$  & $(64, 42, 26, 30)$  &  $\{ [42]^1, [2]^{42}, [-6]^{21} \}$ & 1& 3& $PL_6(8)$ \\ \hline
					$\G(\bm{3},\bm{5^2})$ & $(25, 8, 3, 2)$   &  $\{ [8]^1, [3]^{8}, [-2]^{16} \}$ & 1& 1& $PL_2(5)$ \\ 
					$\bar\G(\bm{3},\bm{5^2})$ & $(25, 16, 9, 12)$   &  $\{ [16]^1, [1]^{16}, [-4]^{8} \}$ & 1& 1& $PL_4(5)$ \\ \hline 
					$\G(\bm{3},\bm{5^4})$  & $(625, 208, 63, 72)$ &  $\{ [208]^1, [8]^{416}, [-17]^{208} \}$ & 1& 2& no \\ 
					$\bar\G(\bm{3},\bm{5^4})$  & $(625, 416, 279, 272)$ &  $\{ [416]^1, [16]^{208}, [-9]^{416} \}$ & 1& 2& no \\ \hline
					$\G(4,3^4)$  & $(81, 20, 1, 6)$ &    $\{ [20]^1, [2]^{60}, [-7]^{20},  \}$ & 1& 2& no \\ 
					$\bar\G(4,3^4)$  & $(81, 60, 45, 42)$ &    $\{ [60]^1, [6]^{20}, [-3]^{60},  \}$ & 1& 2& no \\ \hline
					$\G(4,3^6)$  & $(729, 182, 55, 42)$ &  $\{ [182]^1, [20]^{182}, [-7]^{546} \}$ & 1& 3& $PL_7(27)$ \\ 
					$\bar\G(4,3^6)$  & $(729, 546, 405, 420)$ &  $\{ [546]^1, [6]^{546}, [-21]^{182} \}$ & 1& 3& $PL_{21}(27)$ \\ \hline
					$\G(\bm{4},\bm{7^2})$  & $(49, 12, 5, 2)$    &  $\{ [12]^1, [5]^{12}, [-2]^{36} \}$ & 1& 1& $PL_2(7)$ \\ 
					$\bar\G(\bm{4},\bm{7^2})$  & $(49, 36, 25, 30)$    &  $\{ [36]^1, [1]^{36}, [-6]^{12} \}$ & 1& 1& $PL_6(7)$ \\ \hline
					$\G(\bm{4},\bm{7^4})$  & $(2401, 600, 131, 156)$ & $\{ [600]^1, [12]^{1800}, [-37]^{600} \}$ & 1& 2& no \\ 
					$\bar\G(\bm{4}, \bm{7^4})$  & $(2401, 1800, 1332, 1355)$ & $\{ [1800]^1, [36]^{600}, [-13]^{1800} \}$ & 1& 2& no \\ \hline
					$\G(\bm{5},\bm{3^4})$  & $(81, 16, 7, 2)$    &  $\{ [16]^1, [7]^{16}, [-2]^{64} \}$ & 2& 1& $PL_2(9)$ \\ 
					$\bar\G(\bm{5},\bm{3^4})$  & $(81, 64, 49, 56)$    &  $\{ [64]^1, [1]^{64}, [-8]^{16} \}$ & 2& 1& $PL_8(9)$ \\ \hline
					$\G(\bm{5},\bm{7^4})$  & $(2401,480,119, 90)$ &  $\{ [480]^1, [39]^{480}, [-10]^{1920} \}$ & 2& 1& $PL_{10}(49)$ \\ 
					$\bar\G(\bm{5},\bm{7^4})$  & $(2401,1920,1560, 1529)$ &  $\{ [1920]^1, [9]^{1920}, [-40]^{480} \}$ & 2& 1& $PL_{40}(49)$ \\ \hline
			\end{tabular}
		\end{table}
	\end{small}

We have marked in bold those graphs $\G(k,p^m)$ with $k \ne p^\ell+1$ for some $\ell \mid \frac m2$.
We point out that, for instance, the graphs with $q=7^4$ do not appear in the Brouwer's lists (\cite{Br}) of strongly regular graphs.
		\hfill $\lozenge$
\end{exam}

\section{Ramanujan GP-graphs} \label{sec:6}
If $\Gamma$ is an $n$-regular graph, then $n$ is the greatest eigenvalue of $\G$. 
Recall that a connected $n$-regular undirected graph is Ramanujan if 
\begin{equation} \label{ram cond 2}
	\lambda(\G) \le 2\sqrt{n-1},
\end{equation}
where $\lambda(\G)$ is the maximum absolute value of the non-principal eigenvalues of $\Gamma$
\begin{equation} \label{lambda G}	 
	\lambda(\G) = \max_{\lambda \in \mathrm{Spec}(\Gamma)} \{|\lambda| : |\lambda|\ne n \}. 
\end{equation}
Here we are interested in Ramanujan generalized Paley graphs: 
we will first classify all semiprimitive GP-graphs which are Ramanujan and then show that all GP-graphs $\G(k,q)$ with $1\le k\le 4$ are indeed Ramanujan.

\subsection{All Ramanujan semiprimitive GP-graphs} \label{sec:7.1}	
We recall that semiprimitive graphs are integral and undirected.
We now give a complete characterization of the semiprimitive generalized Paley graphs which are Ramanujan. 
In particular, we will show that if a semiprimitive GP-graph $\G(k,q)$ is Ramanujan then $k\in \{2,3,4,5\}$.

\begin{thm} \label{rama car}
	Let $q=p^m$ with $p$ prime and let $(k,q)$ be a semiprimitive pair.
	Then, the graph $\G=\G(k,p^m)$ is Ramanujan if and only if $\G^+=\G^+(k,p^m)$ is Ramanujan and this happens if and only if
	\begin{enumerate} 
		\item[$(a)$] $\G$ is the classic Paley graph $\G(2,q)$, with $q \equiv 1 \pmod 4$,
	\end{enumerate}
	or $m$ is even and $k, p, m$ are as in one of the following cases:
	\begin{enumerate} 
		\item[$(b)$] $k=3$, $p=2$ and $m\ge 4$.  \sk
		
		\item[$(c)$] $k=3$, $p \ne 2$ with $p\equiv 2 \pmod 3$ and $m\ge 2$. \sk
		
		\item[$(d)$] $k=4$, $p=3$ and $m\ge 4$. \sk 
		
		\item[$(e)$] $k=4$, $p\ne 3$ with $p\equiv 3 \pmod 4$ and $m\ge 2$. \sk 
		
		\item[$(f)$] $k=5$, $p=2$ and $m\ge 8$ with $4\mid m$. \sk 
		
		\item[$(g)$] $k=5$, $p\ne 2$ with $p\equiv 2,3\pmod 5$ and $m\ge 4$ with $4\mid m$. \sk 
		
		\item[$(h)$] $k=5$, $p\equiv 4 \pmod 5$ and $m\ge 2$ even. 
	\end{enumerate} 
	Moreover, $\bar{\G}(k,q)$ is Ramanujan for every semiprimitive pair $(k,q)$.
\end{thm}

\begin{proof}
We begin by noticing that, by Theorem \ref{semiprimitive}, the graphs $\G(k,q)$, $\bar{\G}(k,q)$ and $\G^+(k,q)$ are connected for any semiprimitive pair $(k,q)$, since the multiplicity of the principal eigenvalue is one.
Also, that $\G$ is Ramanujan if and only if $\G^+$ is Ramanujan follows directly from the fact that 
$$\lambda(\G^+(k,q))=\lambda(\G(k,q))$$ 
by Theorem \ref{Spectro Gkq}, and hence \eqref{ram cond 2} holds for both or for none of the graphs.

Now, note that $k=1$ is excluded since $(1,q)$ is not a semiprimitive pair and that $k=2$ corresponds to the classic Paley graph $\G(2,q)$, with $q \equiv 1 \pmod 4$, which is well-known to be Ramanujan (hence $(a)$). 
So it is enough to consider semiprimitive pairs $(k,p^m)$ with $k>2$.
	
	We divide the proof of the characterization of semiprimitive Ramanujan GP-graphs into three steps: in steps 1 and 2 we prove the statement for the graphs $\G(k,q)$, and in step 3 we prove it for the complements $\bar \G(k,q)$.

	\noindent \textit{Step 1.} 
	Here we prove that if $\G$ is Ramanujan with $(k,p^m)$ a semiprimitive pair with $k\ne 2$, then $3\le k\le 5$.
	
	Note that for $k\ge 3$ we have $\lambda(\G)=|\lambda_1|$ (see \eqref{autoval semip} in Theorem \ref{semiprimitive}).
	Since $\G$ is Ramanujan and undirected, 
	\eqref{ram cond 2} reads
	$$\tfrac 1k |\sigma(k-1)p^{\frac m2}-1| \le 2\sqrt{\tfrac{p^m-(k+1)}{k}}.$$ 
	This inequality is equivalent to $(k-1)^2 p^m-2\sigma(k-1) p^{\frac m2}+1\le 4k(p^{m}-(k+1))$  which holds if and only if 
	\begin{equation}\label{rammast2}
	4k(k+1)+1\le p^m (4k-(k-1)^2)+ 2(k-1)\sigma p^{\frac m2}.
	\end{equation}

	Assume first that $\sigma=-1$. Then, \eqref{rammast2} takes the form
	\begin{equation*}\label{rammast}
	2(k-1)p^{\frac m2}+4k(k+1)+1\le p^m (4k-(k-1)^2).
	\end{equation*} 
	Since the left hand side of this inequality is positive, we have that $4k-(k-1)^2>0$, and this can only happen if $ k \le 5$.

	Now, let $\sigma=1$. In this case, inequality \eqref{rammast2} is equivalent to 
	\begin{equation}\label{rammast3}
	0\le (4k-(k-1)^2) p^m+ 2(k-1)p^{\frac m2} -(2k+1)^2.
	\end{equation}
	Suppose $k>5$ and consider the quadratic polynomial 
	$$P_{k}(x)=(4k-(k-1)^2) x^2+ 2(k-1)x -(2k+1)^2.$$ 
	Hence, $P_k(x)$ has negative leading coefficient and its discriminant is given by 
	$$\Delta(k)= 4 \big( (k-1)^2 - ((k-3)^2-8)(2k+1)^2 \big).$$
	Since $(2k+1)^2> (k-1)^2$, the sign of $\Delta(k)$ depends on $(k-3)^2-8$. Since $k>5$, we have that $(k-3)^2-8>0$ and thus $\Delta(k)<0$. So, the quadratic polynomial $P_{k}$ has no real roots and since $P_{k}(0)=-(2k+1)^2<0$, we obtain that $P_{k}(x)<0$ for all $x\in \mathbb{R}$, in particular $P_{k}(p^{\frac m2})<0$ for all $p$ and $m$, contradicting \eqref{rammast3}. 
	Therefore, if $\G$ is Ramanujan then $k\le 5$, as we wanted to show.

	\noindent \textit{Step 2.} 
	We now show that the pair $(k,q)$ semiprimitive with $k\le 5$ can only happen as stated in the theorem; and, in these cases, $\G(k,q)$ is Ramanujan.

	As mentioned at the beginning, the case $k=1$ is excluded and $k=2$ corresponds to the classic Paley graph, which is Ramanujan.
	If $k=3$, then necessarily $p\equiv 2 \pmod 3$ and $m$ is even, for if not the pair $(k,p^m)$ is not semiprimitive. 
	In this case, \eqref{rammast2} is given by $49\le 8 p^m+4\sigma p^{\frac m2}$.
	The worst possibility is when $\sigma =-1$, and in this case the previous inequality reads 
	$$12+\tfrac{1}{4}\le p^{\frac m2}(2p^{\frac m2}-1).$$
	This clearly holds if and only if $p$ is odd and $m\ge 2$ or $p=2$ and $m\ge 4$, and thus $\G(3,p^m)$ is Ramanujan in 
	these cases. This proves $(b)$ and $(c)$.

	If $k=4$, then we must have $p\equiv 3 \pmod 4$ and $m$ is even, for if not the pair $(k,p^m)$ is not semiprimitive. 
	In this case, \eqref{rammast2} is given by $81\le 7p^m+6\sigma p^{\frac m2}$. 
	As before, the worst case is when $\sigma =-1$, and thus the inequality is equivalent to 
	$$11+ \tfrac{4}{7}\le p^{\frac m2}(p^{\frac m2}-\tfrac{6}{7}) .$$
	This holds if and only if $p>3$ and $m\ge 2$ or $p=3$ and $m\ge 4$ and hence $\G(4,p^m)$ is Ramanujan in these cases, thus showing $(d)$ and $(e)$.

	In the last case, if $k=5$, then $(5,p^m)$ is semiprimitive if and only if $p\equiv 2,3 \pmod 5$ and $4\mid m$ or 
	else $p\equiv 4 \pmod 5$ and $m\ge 2$ even. 
	On the other hand, in this case \eqref{rammast2} is given by $121 \le 4p^m+8\sigma p^{\frac m2}$,
	which is equivalent to 
	$$30+\tfrac{1}{4}\le p^{\frac m2}(p^{\frac m2}+2\sigma).$$
	If $p=2$, then necessarily $m\ge 8$ since $4\mid  m$ and $m=4$ does not satisfy the above inequality.
	Clearly, the inequality holds for $p\equiv 2,3 \pmod 5$ with $p\ne 2$ and $4\mid m$.
	Finally, notice that the right hand side of the inequality increases when $p$ increases. 
	The first prime $p$ with $p\equiv 4 \pmod 5$ is $p=19$ that clearly satisfies the inequality for $m\ge 2$ even, 
	so we obtain that the inequality holds for all of primes $p\equiv 4 \pmod 5$ with $m\ge 2$ even.
	In this way we have shown that $\G(5,p^m)$ is Ramanujan in all the cases in the statement, proving items $(f)$--$(h)$.

	\noindent \textit{Step 3.} 
	Now, we consider the complementary graphs $\bar{\Gamma}=\bar{\G}(k,q)$. We have that 
	$$\lambda(\bar{\G}) = |\bar{\lambda}_1|=(k-1) \tfrac{p^{\frac m2}+\sigma}{k}$$ 
	and the regularity degree of $\bar{\G}$ is $n(k-1)$. 
	Notice that we can assume that $k>2$, since $k=2$ correspond to the classic Paley graph which is self-complementary, and hence Ramanujan. 
	Also, without loss of generality we can assume that $\sigma=1$. 
	Inequality \eqref{ram cond 2} becomes 
	\begin{equation} \label{ram comp}
	(k-1)\tfrac{p^{\frac m2}+1}{k}\le 2\sqrt{\tfrac{(p^m-1)(k-1)-k}{k}}
	\end{equation}
	which is equivalent to 
	$(k-1)^{2}p^m+2(k-1)^2p^{\frac m2}+(k-1)^2 \le 4k(p^m(k-1)-(2k-1))$
	and therefore we have   
	$$2(k-1)^2p^{\frac m2}+(k-1)^2+4k(2k-1)\le p^{m}(4k(k-1)-(k-1)^2).$$
	Notice that $(k-1)^2+4k(2k-1)=(3k-1)^2$ and $4k(k-1)-(k-1)^2=(k-1)(3k+1)$. 
	Let us consider the quadratic polynomial 
	$$Q_{k}(x)=x^2-b_{k} x-c_k \qquad \text{where} \qquad 
	b_{k}=\tfrac{2(k-1)}{(3k+1)} \quad \text{and} \quad c_{k}=\tfrac{(3k-1)^2}{(k-1)(3k+1)}.$$
	Hence, $\bar \G(k,q)$ is Ramanujan if and only \eqref{ram comp} holds, that is if and only if $Q_k(p^{\frac m2})>0$.
	
	Clearly $b_{k}<1$ and $4c_{k}<15$, this implies that the greatest real root $r$ of $Q_k$ satisfies 									
	$$r = \tfrac{b_{k}}{2} + \tfrac{1}{2} \sqrt{b_k^2+4c_k} < \tfrac 12 +2 <3.$$
	Since $(k,p^m)$ is a semiprimitive pair and $k>2$, we have that $p^{\frac m2}\ge 3$. 
	This implies that $Q_{k}(p^{\frac m2})>0$ since $Q_{k}$
	has a positive leading coefficient. Therefore $\bar \G(k,p^m)$ is Ramanujan for all semiprimitive pair $(k,p^m)$ with $\sigma=1.$
	The case $\sigma=-1$ can be proved analogously.
\end{proof}

The previous result gives the following eight infinite families of Ramanujan semiprimitive GP-graphs: 
\begin{enumerate}[$(a)$]
	\item \: $\{ \G(2,q)\}$ with 
	$q\equiv 1 \pmod 4$, i.e.\@ the classic Paley graphs, 
	
	\item \: $\{ \G(3,4^{t})\}_{t\ge 2}$, 
	
	\item \: $\{ \G(3,p^{2t})\}_{t\ge 1}$ with $p \equiv 2 \pmod 3$ and $p\ne 2$, 
	
	\item \: $\{ \G(4,9^{t})\}_{t\ge 2}$, 
	
	\item \: $\{ \G(4,p^{2t})\}_{t\ge 1}$ with $p \equiv 3 \pmod 4$ and $p\ne 3$, 
	
	\item \: $\{ \G(5,16^{t})\}_{t\ge 2}$, 
	
	\item \: $\{ \G(5,p^{4t})\}_{t\ge 1}$, with $p \equiv 2,3 \pmod 5$ and $p\ne 2$, 
	
	\item \: $\{ \G(5,p^{2t})\}_{t\ge 1}$ with $p \equiv 4 \pmod 5$.
\end{enumerate}
Note that five of them are valid for an infinite number of primes.
The smallest graphs in each family are: 
$$\G(2,5), \: \G(3,16), \: \G(3,49), \: \G(4,81), \: \G(4,49), \: \G(5,256), \: \G(5,81), \: \text{and} \: \G(5,361),$$
%$\G(2,5)$, $\G(3,16)$, $\G(3,49)$, $\G(4,81)$, $\G(4,49)$, $\G(5,256)$, $\G(5,81)$ and $\G(5,361)$, 
respectively. Also, notice that all the graphs in Table \ref{tablita srg} are Ramanujan, corresponding to the families ($c$), ($e$) and ($g$).

\begin{rem}
	($i$) The Ramanujan GP-graphs $\G(k,q)$ with $k=p^\ell+1$ are characterized in Theorem~8.1 in \cite{PV1}. 
	There, we proved that 
		$$\Gamma_{q,m}(\ell) = \G(p^{\ell}+1,p^m),$$ 
	with $\ell \mid m$ such that $m_\ell$ even 
	and $\ell \ne \frac m2$, is Ramanujan if and only if $q=2,3,4$ with $\ell=1$ and $m\ge 4$ even. 
	This says that $\G(p^{\ell}+1,p^m)$ is Ramanujan only in the cases ($b$), ($d$) and ($f$),  
	giving the infinite families  
	$$\{ \G(3, 4^t) \}_{t\ge 2}, \qquad \{ \G(4, 9^t)\}_{t\ge 2} \qquad \text{ and } \qquad \{\G(5, 16^t)\}_{t\ge 2}$$
	of Ramanujan graphs. 
	The first two families coincide with those in ($b$) and ($d$), while the third one gives just half the graphs in ($f$), precisely those with $t$ even in ($f$). 
	Thus, the last proposition extends this characterization of Ramanujan GP-graphs $\G(q^\ell+1,q^m)$ to all semiprimitive pairs $(k,p^m)$,  that is in the case $q=p$. 
	
	($ii$) 
	When $p=2$, the last proposition gives nothing new, since the possible values of $k \in \{2,3,4,5\}$ such that $(k,2^m)$ is a semiprimitive pair reduces to $k=3, 5$, which corresponds to the cases $p=2$ with $\ell=1, 2$ in ($i$) above.
\end{rem}

\begin{exam}
	From Theorem \ref{rama car}, the following GP-graphs are Ramanujan:
	$$\begin{tabular}{|p{1cm}|p{9cm}|} \hline 
	$p=2$ & $\G(3,16)$, $\G(3,64)$, $\G(3,256)$, $\G(5, 256)$, \\ \hline 
	$p=3$ & $\G(2,81)$, $\G(4,81)$, $\G(\textbf{5}, \textbf{81})$, 
	$\G(2,729)$, $\G(4,729)$, $\G(2,6{.}561)$, $\G(4,6{.}561)$, $\G(\textbf{5}, \textbf{6{.}561})$ \\ \hline
	$p=5$ & $\G(2,25)$, $\G(\textbf{3},\textbf{25})$, $\G(2,625)$, $\G(\textbf{3},\textbf{625})$, 
	$\G(2,15{.}625)$, $\G(\textbf{3},\textbf{15{.}625})$, $\G(2,390{.}625)$, $\G(\textbf{3},\textbf{390{.}625})$ \\ \hline  
	$p=7$ & $\G(\textbf{4},\textbf{49})$, $\G(\textbf{4},\textbf{2{.}401})$, $\G(\textbf{5},\textbf{2{.}401})$, 
	$\G(\textbf{4},\textbf{117{.}649})$, $\G(\textbf{4},\textbf{5{.}764{.}801})$, $\G(\textbf{5},\textbf{5{.}764{.}801})$ \\ \hline  
	\end{tabular}$$
	where $6{.}561=3^8$, $15{.}625=5^6$, $390{.}625=5^8$, $7^4=2{.}401$, $7^6=117{.}649$ and $7^8=5{.}764{.}801$. 
	We have marked in bold those graphs $\G(k,p^m)$ with $k>2$ and $k\ne p^\ell+1$ for some $\ell \mid \frac m2$.
	\hfill $\lozenge$
\end{exam}

\subsection{Ramanujan graphs $\G(k,q)$ with $1\le k \le 4$} \label{subsec: Ram G3G4}
It is well-known that the complete graphs $K_n$ and the classic Paley graphs $P(q)$ with $q\equiv 1 \pmod 4$ are Ramanujan (it is immediate to check it from \eqref{ram cond 2} and Examples \ref{g1q} and \ref{Ex:Paley}). 

In general, $\G(k,q)$ can be directed.  
There are (to the authors knowledge) two notions of Ramanujan $n$-regular digraphs. 
We recall that a directed graph is $n$-regular if its in-degree and out-degree are both equal to $n$. 
A connected $n$-regular undirected graph is \textit{Ramanujan} if it satisfies \eqref{ram cond 2}, that is 
\begin{equation*}
\lambda(\G) \le 2\sqrt{n-1},
\end{equation*}
where $\lambda(\G)$ is the maximum absolute value of the non-principal eigenvalues of $\Gamma$. 
An $n$-regular connected directed graph $\G$ is Ramanujan if it satisfies \eqref{ram cond 2} and also its adjacency matrix can be diagonalized by a unitary matrix, see for instance \cite{L}. %(\textit{we will say Ramanujan in the classic sense}).
A more recent definition due to Lubotzky and Parzanchevski (see \cite{LuPa}, \cite{Pa}) is that an $n$-regular connected digraph $\G$ is Ramanujan if 
\begin{equation*} \label{eq: Ram dir}
	\lambda(\G) \le \sqrt n	.
\end{equation*}
One can check from the spectrum given in \eqref{spec paley dir} that the directed Paley graphs $\vec P(q)$, i.e.\@ those $\G(2,q)$  with $q \equiv 3 \pmod 4$, are Ramanujan under the two notions
(in the second sense, it is mentioned in \S 13.3.3 in \cite{Pa} for $q=p$ prime). 
In this way we have that the graphs $\G(1,q)$ and $\G(2,q)$ are Ramanujan for any $q$. 

We will now study which GP-graphs $\G(3,q)$ and $\G(4,q)$ are Ramanujan. In Theorem \ref{rama car} we have found those semiprimitive GP-graphs $\G(3,q)$ and $\G(4,q)$ which are Ramanujan (they are the graphs given in $(b)$--$(e)$ in the previous list). 
For this reason, we now assume that $\G(k,q)$ is non-semiprimitive for $k=3,4,$ where $q=p^m$, that is to say $p \equiv 1 \pmod k$.

To study the Ramanujanicity of the graphs $\G(3,q)$ we will need a lemma. So, we first fix some notations.
Denote by $\lambda_0,\lambda_1,\lambda_2$ the non-principal eigenvalues of $\G(3,q)$, 
that is (see $(a)$ in Theorem \ref{gp3q})
\begin{equation}\label{autoval G3}
\lambda_0 = \tfrac{a\sqrt[3]{q}-1}{3}, \qquad \lambda_1 = \tfrac{-\frac{1}{2} (a+9b)\sqrt[3]{q}-1}{3} \qquad \text{and} \qquad 
\lambda_2 = \tfrac{-\frac 12 (a-9b) \sqrt[3]{q}-1}{3},
\end{equation}
where $a,b$ are integers uniquely determined by the conditions \eqref{ab27}.
Notice that $\lambda(\G(3,q))$ can be realized by any of the three non-principal eigenvalues of $\G(3,q)$.

\begin{lem}\label{lem signs}
	Let $q=p^{3t}$ for some $p$ prime such that $p\equiv 1\pmod{3}$ and let $\lambda_1$ and $\lambda_2$ be 
	as in \eqref{autoval G3}. 
	Then, we have: 
	\begin{enumerate}[$(a)$]
		\item If $|\lambda_1|>|\lambda_2|$, then $a$ and $b$ have the same sign. \sk 
		\item If $|\lambda_2|>|\lambda_1|$, then $a$ and $b$ have different signs.
	\end{enumerate}
\end{lem}

\begin{proof}
($a$) In this case, $|\lambda_1|>|\lambda_2|$ is equivalent to $\tfrac 12 (a+9b)p^t+1 > \tfrac 12 (a-9b)p^t+1 $. This implies that 
	$$\big( (a+9b)p^t+2 \big)^2 > \big( (a-9b)p^t+2 \big)^2,$$
from which after some computations we obtain that 
	$$b(2+ap^t)>0.$$
Hence, $b$ and $2+ap^t$ have the same sign. Taking into account that $2+ap^t$ has the same sign as $ap^t$, since $p\ge 7$, 
and so the same sign of $a$, we see that $a$ and $b$ have the same sign. 
	
\noindent ($b$) 
This case can be proved in the same way as ($a$), by noticing that $|\lambda_1|<|\lambda_2|$ implies that
$b (2+ap^t)<0$.
\end{proof}

We now show that any non-semiprimitive GP-graph $\G(3,q)$  and any non-semiprimitive GP-graph $\G(4,q)$ with $q$ a square are Ramanujan.

\begin{thm}\label{teo G3 Ram}
Let  $k \in \{3,4\}$ and  $q=p^{m}$ for some prime $p$ and $m\in \N$. 
If $p\equiv 1 \pmod{k}$, with $m$ even if $k=4$, then $\G(k,q)$ is Ramanujan.
\end{thm}

\begin{proof}
Notice that the hypothesis $p\equiv 1 \pmod{k}$ implies that $\G(k,q)$ is well defined and it is non-semiprimitive. 
We divide the proof in two parts, one for $\G(3,q)$ and one for $\G(4,q)$.
By Theorems \ref{gp3q} and \ref{gp4q} we know that $\G(3,q)$ and $\G(4,q)$ are both undirected graphs in all the cases. Thus,
by \eqref{ram cond 2}, $\Gamma$ is Ramanujan if and only if $\lambda(\G)\le 2\sqrt {n-1}$.

\noindent \textit{$\bullet$ The graphs $\G(3,q)$}. We begin by showing that $\G(3,q)$ in the non-semiprimitive case is Ramanujan. 
Let $p\equiv 1 \pmod 3$, $k=3$, and $n=\frac{q-1}3$ where $q=p^{m}$. 

\sk 

$(i)$ Suppose that $m=3t$ for some $t\in \N$ (so we are in case $(a)$ of Theorem \ref{gp3q}). 
Assume first that 
	$$\lambda(\G(3,q)) = |\lambda_0| = |\tfrac{ap^t-1}{3}|.$$
Notice that $\lambda(\G(3,q)) \le 2 (\tfrac{q-1}{3}-1)^{\frac 12}$ is equivalent to
	$$|ap^t-1| \le 2\sqrt{3} \sqrt{p^{3t}-4}.$$
Since $|ap^t-1|=|a|p^{t}\pm 1$, after some computations, we have that the above inequality is equivalent to
	$$(|a|p^{t}\pm 1)^2\le 12 (p^{3t}-4).$$
By taking into account that $a^2\le 4p^t$, we have that 
	$$(|a|p^{t}\pm 1)^2\leq 4p^{3t}+2|a|p^{t}+1.$$ 
Finally, since $p\equiv 1 \pmod{3}$ then $p\ge 7$ and so we have that 
$4p^{3t}+2|a|p^{t}+1\le 12 (p^{3t}-4)$ is always true which implies that $\G(3,q)$ is Ramanujan in this case.
	
Now assume that 
	$$\lambda(\G(3,q)) = |\lambda_1| = |\tfrac{-\frac{1}{2}(a+9b)p^t-1}{3}|.$$ 
In this case, $\lambda(\G(3,q))\le 2 (\frac{q-1}{3}-1)^{\frac 12}$ is equivalent to
	$$|\tfrac{-(a+9b)p^t-2}{6}|\leq \tfrac{2\sqrt{3}}{3}\sqrt{p^{3t}-4}.$$  
Since $a+9b\neq 0$ and $p\ge 7$, we have that $|-(a+9b)p^t-2|=|a+9b|p^t \pm 2$, 
and so the above inequality is equivalent to
	$$(|a+9b|p^t \pm 2)^2\leq 48(p^{3t}-4).$$
By Lemma \ref{lem signs}, the integers $a$ and $b$ have the same sign and so $(a+9b)^2= a^2+18 |a| |b|+81b^2$, which implies that
	$$(|a+9b|p^t \pm 2)^2\le (a^2+18|a||b|+81b^2)p^{2t}+ 4 |a+9b|p^t +4.$$
By taking into account that $4p^t=a^2+27b^2$ we have that 
$$
|a|\le 2 p^{\frac{t}{2}} \quad \text{and} \quad |b| \le \tfrac{2\sqrt{3}}{9} p^{\frac{t}{2}}
$$ 
which implies that 
$$
a^2+18|a||b|+81b^2  \le 12 p^{\frac{t}{2}}+8\sqrt{3} p^{\frac{t}{2}}=4(3+2\sqrt{3}) p^{\frac{t}{2}}\le 28p^{\frac{t}{2}}.
$$
Now, since $a$ and $b$ have the same sign and $4p^t=a^2+27b^2$, we have that
	$$|a+9b|= |a|+9|b|\le 2(1+\sqrt{3}) p^{\frac{t}{2}}\le 6 p^{\frac{t}{2}}.$$
In this way we obtain that 
	$$(a^2+18|a||b|+81b^2)p^{2t}+ 4 |a+9b|p^t +4\le 28 p^{3t}+ 24 p^{\frac{3t}{2}}+4.$$
Finally, since $p\ge 7$, we have that $28 p^{3t}+ 24p^{\frac{3t}{2}}+4\le 48(p^{3t}-4)$, which implies that $\G(3,q)$ is Ramanujan.

The remaining case,  $\lambda(\G(3,q))=|\lambda_2|$, can be proved in a similar way, using item ($b$) instead of ($a$) of the above Lemma.

$(ii)$ Now assume that $3\nmid m$ (so we are in case $(b)$ of Theorem \ref{gp3q}). In this case, the eigenvalues of $\G(3,q)$ are given by
			$$x_j= -\tfrac 13 \big( 1+\omega^j C+ \tfrac{q}{\omega^j C} \big) \qquad \text{with} \qquad 
			  C = \sqrt[3]{q} \sqrt[3]{\tfrac 12 (-a_0 + i3\sqrt 3 b_0)}$$			
for $j=0,1,2$ where $\omega=e^{\frac{2\pi i}{3}}$ and $a_0, b_0$ are integers satisfying 
$4q=a_0^2+27b_0^2$, $a\equiv 1 \pmod 3$ and $(a_0,p)=1$.
Notice that $|C|=\sqrt{q}$. In fact, 
	$$|C| = \sqrt[3]q \, |\tfrac{a_0+i\sqrt{27}b_0}2|^{\frac 13}= \sqrt[3]q \, \big(\sqrt{\tfrac{a_0^2+ 27b_0^2}4}\big)^{\frac 13}=
	\sqrt[3]q \sqrt[3]{\sqrt{q}} = \sqrt q.$$
Thus, for any $j\in \{0,1,2\}$ we have that  
	$|x_j|\le \tfrac{1+2\sqrt{q}}{3}$. 
It is straightforward to see that 
	$$\tfrac{1+2\sqrt{q}}{3}\le 2\sqrt{n-1}$$ 
holds for any $q\ge 7$, since it is equivalent to $8q-4\sqrt q -49 \ge 0$. Therefore, \eqref{ram cond 2} holds, and $\G(3,q)$ is Ramanujan in this case.

Thus, we have seen that any non-semiprimitive GP-graph $\G(3,p^{m})$ with $m\in \N$ (i.e.\@ with $p\equiv 1 \pmod 3$) is Ramanujan.

\sk 

\noindent \textit{$\bullet$ The graphs $\G(4,q)$}. 
We now look at the non-semiprimitive GP-graphs $\G(4,q)$. So, let $k=4$, $n=\frac{q-1}4$ and $q=p^m$  for some $m\in \N$ even. 

$(i)$ Suppose first that $4\mid m$.
By $(a)$ in Theorem~\ref{gp4q}, the spectrum of $\G(4,p^{4t})$ is given by
	$$\big\{ [n]^1, \big[\tfrac{p^{2t} + 4dp^{t}-1}{4}\big]^n, \big[\tfrac{p^{2t} - 4dp^t-1}{4}\big]^n, 
	\big[\tfrac{-p^{2t} + 2c p^{t}-1}{4}\big]^n, \big[\tfrac{-p^{2t} - 2c p^{t}-1}{4}\big]^n \big\}$$
where $n=\frac{p^{4t}-1}{4}$ and $c,d$ are integers uniquely determined by $p^{2t}=c^2+4d^2$, $c\equiv 1 \pmod 4$ and $(c,p)=1$.
In particular we have that
	\begin{equation}\label{eq |c|d|}
		|c|\le p^{t} \qquad \text{and}  \qquad |d|\le \tfrac 12 p^t.
	\end{equation}
By a simple comparison of the eigenvalues, we have that 
	\begin{equation*}
		\lambda(\G(4,p^{4t}))= \max \Big\{ \frac{p^{2t} + 4|d|p^t-1}{4}, \frac{p^{2t} - 2|c|p^t+1}{4} \Big\}.
	\end{equation*}
In this case the inequality $\lambda(\G(4,p^{4t}))\le 2\sqrt{n-1}$ is equivalent to
	\begin{equation}\label{ram}
		\lambda(\G(4,p^{4t}))^{2}\le p^{4t}-5.
	\end{equation}
It is enough to see what happens in any possible case.
	
Suppose first that $\lambda(\G(4,p^{4t}))= \tfrac14 (p^{2t} + 4|d|p^t-1)$, in this case the equation \eqref{ram} turns into
	$$(p^{2t} + 4|d|p^t-1)^2\le 16p^{4t}-80.$$
Clearly, $$(p^{2t} + 4|d|p^t-1)^2\le (p^{2t} + 4|d|p^t)^2+1= p^{4t}+8|d|p^{3t}+ 16d^2 p^{2t}+1.$$
By \eqref{eq |c|d|} we have that 
	$8|d|p^{3t}+ 16d^2 p^{2t}\le 4p^{4t}+4p^{4t}=8p^{4t}$ 
and hence 
	$$(p^{2t} + 4|d|p^t-1)^2\le 9 p^{4t}+1,$$
since $p\ge 5$, we have that $9 p^{4t}+1\le 16p^{4t}-80$ is true and so $\G(4,p^{4t})$ is Ramanujan as desired.

On the other hand, when $\lambda(\G(4,p^{4t})) = \tfrac 14 (p^{2t} - 2|c|p^t+1)$, by \eqref{eq |c|d|} we have that 
	$$(p^{2t} - 2|c|p^t-1)^2\le (p^{2t}+1)^2+4c^2p^{2t}\le 4p^{4t}+4p^{4t}=8p^{4t}\le 16p^{4t}-80.$$
A similar argument as above allows us to conclude that $\G(4,p^{4t})$ is Ramanujan as asserted. 
	
$(ii)$ Now, if $m\equiv 2\pmod{4}$, proceeding similarly as in the previous case, by $(b)$ in Theorem~\ref{gp4q} we have that 
	\begin{equation*}
		\lambda(\G(4,p^{4t+2}))= \max \Big\{ \tfrac{p^{2t+1} +1 + \sqrt{2(p^{4t+2}+cp^{2t+1})}}{4}, \tfrac{p^{2t+1} -1 + \sqrt{2(p^{4t+2}-cp^{2t+1})}}{4} \Big\}
	\end{equation*}
where  $n=\frac{p^{4t+2}-1}{4}$ and $c,d\in  \Z$ are uniquely determined by 
$p^{2t+1}=c^2+4d^2$, $c\equiv 1 \pmod 4$ and $(c,p)=1$.
 
A similar computation as in the case $m\equiv 0 \pmod 4$ shows that $\G(4,p^{4t+2})$ is Ramanujan (we leave the details to the reader), and the result is thus proved. 
\end{proof}

To conclude, we now summarize the results of this section on the Ramanujanicity of the GP-graphs $\G(k,q)$ with 
$1\le k\le 5$.
\begin{rem}
Consider $\G(k,q)$ where $1\le k\le 5$ and $k\mid q-1$ with $q=p^m$ with $p$ prime. Then we have the following:
\begin{enumerate}[$(a)$]
	\item The graphs $\G(1,q)$ and $\G(2,q)$ are all Ramanujan. \sk 
	
	\item In the semiprimitive case, the graphs $\G(k,q)$ with $k\in \{3,4\}$ are Ramanujan if and only if 
		$\G(k,p^m)$ with $p\equiv -1\pmod k$ and $m\ge 2$, where $m \ge 4$ if $p=k-1$ (see Theorem~\ref{rama car}). \sk 
	
	\item In the non-semiprimitive case, the graphs $\G(3,q)$ with $q$ arbitrary and $\G(4,q)$ with $q$ a square are all Ramanujan (see Theorem \ref{teo G3 Ram}).
	
	\item The graph $\G(5,q)$ in the semiprimitive case is Ramanujan if and only if $m$ is even and either $p \equiv 2,3 \pmod 5$, where $m=4t$ ($t\ge 2$ if $p=2$), or else $p \equiv 4 \pmod 5$, by items $(f)$--$(h)$ in Theorem \ref{rama car}. 
	
	\item In all the previous cases ($a$)--($d$), if $\G(k,q)$ is Ramanujan with $1\le k\le 5$, then $\G^+(k,q)$ is also Ramanujan, since $\lambda(\G^+) =\lambda(\G)$, see \eqref{lambda G}, by Theorem \ref{Spectro Gkq} and hence \eqref{ram cond 2} holds.
\end{enumerate}
\end{rem}

Relative to the GP-graphs with $k=5$, the Ramanujanicity of the non-semiprimitive case is open. On the one hand, the spectrum of $\G(5,q)$ for $p \equiv 2,3 \pmod 5$ is still unknown (when the graph is not semiprimitive). For $p\equiv 1 \pmod 5$ with $q=p^{5t}$, the spectrum $\mathrm{Spec}(\G(5,q))$ is given in Proposition \ref{G5}.
For instance, it is immediate to check that the graph $\G(5,11^5)$ in Example \ref{Ex:G5} is Ramanujan. 
In general, we leave the following question: \textit{which graphs $\G(5,p^{5t})$ with $p\equiv 1 \pmod 5$ are Ramanujan? }


\begin{thebibliography}{XXX}
	
\bibitem{Br}{\sc A.E.\@ Brouwer},
\textrm{Strongly regular graphs' page} \verb|www.win.tue.nl/~aeb/graphs/srg/srgtab.html|.


\bibitem{BrV} {\sc A.E.\@ Brouwer, H.\@ Van Maldeghem}, 
\textrm{Strongly regular graphs}, 
\textit{Cambridge University Press} Vol.\@ 182 2022.


\bibitem{BWX}{\sc A.E.\@ Brouwer, R.M.\@ Wilson, Q.\@ Xiang},
\textrm{Cyclotomy and Strongly Regular Graphs},
\textit{J.\@ Algebr.\@ Comb.\@} \textbf{10} (1999), 25--28.


\bibitem{CvL} \textsc{P.J.\@ Cameron, J.H.\@ van Lint}, 
\textrm{Designs, graphs, codes and their links}, 
\textit{Cambridge University Press}, LMSST 22, 1991.


\bibitem{CDS} \textsc{D.\@ Cvetkovi\v c, M.\@ Doob, H.\@ Sachs},
\textrm{Spectra of Graphs -- Theory and Application}, 
\textit{Academic Press, New York}, 1980.


\bibitem{CvP} \textsc{D.\@ Cvetkovi\v c, M.\@ Petri\v c}, 
\textrm{A table of connected graphs of six vertices}, 
\textit{Discrete Math.\@} \textbf{50} (1984), 37--49.


\bibitem{DY}{\sc C.\@ Ding, J.\@ Yang}, 
\textrm{Hamming weights in irreducible cyclic codes},
\textit{Discrete Math.\@}  \textbf{313:4} (2013), 434--446.


\bibitem{Gu1} \textsc{S.\@ Gurak}, 
\textrm{Period polynomials for $\ff_{q^2}$ of fixed small degree}, 
\textit{In Finite Fields and Applications. Proceedings of the fifth international conference on finite fields and applications $F_{q^5}$, University of Augsburg, Germany, August 2--6, 1999. Berlin: Springer}. 196--207 (2001).


\bibitem{Gu2} \textsc{S.\@ Gurak}, 
\textrm{Period polynomials for $\ff_q$ of fixed small degree}, 
\textit{CRM Proc.\@ Lect.\@ Notes} \textbf{36} (2004) 127--145.


\bibitem{HS}\textsc{F.\@ Harary, A.J.\@ Schwenk}, 
\textrm{Which graphs have integral spectra?}, 
\textit{Graphs and Combin.\@, Proc.\@ Capital Conf.\@, Washington D.C.\@ 1973, Lect.\@ Notes Math.\@} \textbf{406} (1974) 45--51.


\bibitem{HLW}\textsc{S.\@ Hoory, N.\@ Linial, A.\@ Wigderson},
\textrm{Expander graphs and their applications},
\textit{Bull.\@ Amer.\@ Math.\@ Soc.\@} \textbf{43:4} (2006), 439--561.


\bibitem{Ho}\textsc{A.\@ Hoshi},
\textrm{Explicit lifts of quintic Jacobi sums and period polynomials for $\ff_{q}$},
\textit{ Proc. Japan Acad.\@} \textbf{82:7} Ser. A (2006) 87--92.


\bibitem{L}\textsc{K.\@ Feng, W.W.\@ Li}, 
\textrm{Character sums and abelian Ramanujan graphs},
\textit{J.\@ Number Theory} \textbf{41:2}, (1992) 199--217.


\bibitem{LP}{\sc T.K.\@ Lim, C.\@ Praeger}, 
\textrm{On Generalised Paley Graphs and their automorphism groups},
\textit{Michigan Math.\@ J.\@} \textbf{58} (2009), 294--308.


\bibitem{Lub}{\sc A.\@ Lubotzky}, 
\textrm{Expander graphs in pure and applied mathematics},
\textit{Bull.\@ Amer.\@ Math.\@ Soc.\@} \textbf{49} (2012), 113--162.


\bibitem{LuPa}{\sc A.\@ Lubotzky, O.\@ Parzanchevski}, 
\textit{Ramanujan graphs to Ramanujan complexes},
\textit{Phil.\@ Trans.\@ R.\@ Soc.\@ A.
Math. Phys. Eng. Sci.} \textbf{378:2163}, Article ID 20180445, 9 p. 


\bibitem{Mur}\textsc{M.\@ Ram Murty},
\textrm{Ramanujan Graphs},
\textit{J.\@ Ramanujan Math.\@ Soc.\@} \textbf{18:1} (2003), 1--20.


\bibitem{My}\textsc{G.\@ Myerson},
\textrm{Period polynomials and Gauss sums},
\textit{Acta Arith.\@} \textbf{39:3} (1981) 251--264.


\bibitem{Pa} \textsc{O.\@ Parzanchevski},
\textrm{Ramanujan graphs and digraphs},
\textit{Lond.\@ Math.\@ Soc.\@ Lect.\@ Note Ser.} \textbf{461}, 344--367.


\bibitem{PV1} \textsc{R.A.\@ Podest\'a, D.E.\@ Videla},
\textrm{Spectral properties of generalized Paley graphs of $(q^\ell+1)$-th powers and applications},
\textit{Discrete Math.\@ Algorithms Appl.\@} (2024) Online first:
\url{https://doi.org/10.1142/S1793830924500563}


\bibitem{PV2} \textsc{R.A.\@ Podest\'a, D.E.\@ Videla},
\textrm{Weight distribution of cyclic codes defined by quadratic forms and related curves},
\textit{ Rev.\@ Unión Mat.\@ Argent.\@} \textbf{62:1} (2021) 219--242.


\bibitem{PV6} \textsc{R.A.\@ Podest\'a, D.E.\@ Videla},
\textrm{The Waring's problem over finite fields through generalized Paley graphs},
\textit{Discrete Math.\@} \textbf{344}, (2021) 112324.


\bibitem{PV5} \textsc{R.A.\@ Podest\'a, D.E.\@ Videla},
\textrm{Integral equienergetic non-isospectral unitary Cayley graphs},
\textit{Linear Algebra Appl.\@} \textbf{612}, (2021) 42--74.


\bibitem{PV3} \textsc{R.A.\@ Podest\'a, D.E.\@ Videla},
\textrm{Generalized Paley graphs equienergetic with their complements},
\textit{Linear Multilinear Algebra}, \textbf{72:3} (2024) 488--515. 


\bibitem{PV8} \textsc{R.A.\@ Podest\'a, D.E.\@ Videla}.
\textit{On regular graphs equienergetic with their complements}. 
Linear Multilinear Algebra \textbf{71:3}, (2023) 422--456.


\bibitem{PV7} \textsc{R.A.\@ Podest\'a, D.E.\@ Videla},
\textrm{Spectral properties of generalized Paley graphs and their associated irreducible cyclic codes},
In progress.


\bibitem{SW}
\textsc{B.\@ Schmidt, C.\@ White},
\textrm{All two weight irreducible cyclic codes?},
\textit{Finite Fields Appl.\@} \textbf{8} (2002), 1--17.
\end{thebibliography}
\end{document}